\documentclass[a4paper,12pt]{article}
\date{}

\usepackage{mathrsfs}
\usepackage[latin1]{inputenc}
\usepackage[english]{babel}
\usepackage{mathrsfs}
\usepackage{amsmath}
\usepackage{amsfonts}
\usepackage{amsthm}

\usepackage{amssymb, graphicx, epsfig}

\usepackage{latexsym}
\DeclareMathOperator*{\sumx}{\sum\nolimits^{*}}

\newcommand{\radsumma}[2]{\genfrac{}{}{0pt}{}{#1}{#2}}

\newtheorem*{ass}{(A1)}
\newtheorem*{ass2}{(A2)}
\newtheorem*{ass3}{(A3)}

\newtheorem{thm}{Theorem}[section] 
\newtheorem{Lemma}{Lemma}[section]
\newtheorem{cor}{Corollary}[section]

\theoremstyle{definition}
\newtheorem{rem}{Remark}[section]
\author{Cecilia~Holmgren}
\title{Novel Characteristics of Split Trees by use of Renewal Theory}

\begin{document}
\maketitle
\begin{abstract}
\emph{We investigate characteristics of random split trees introduced by Devroye \cite{devroye3};
split trees include for example binary search trees, $m$-ary search trees, quadtrees, median of $(2k+1)$-trees, simplex trees, tries and digital search trees. More precisely: We introduce the use of renewal theory in the studies of split trees, and use this theory to prove several results about split trees. A split tree of cardinality $n$ is constructed by distributing $n$ ``balls'' (which often represent ``key numbers'') in a subset of vertices of an infinite tree. One of our main results is to give a relation between the deterministic number of balls $n$ and the random number of vertices $N$. In \cite{devroye3} there is a central limit law for the depth of the last inserted ball so that most vertices are close to $\frac{\ln n}{\mu}+\mathcal{O}\Big(\sqrt{\ln n}\Big)$, where $\mu$ is some constant depending on the type of split tree; we sharpen this result by finding an upper bound for the expected number of vertices with depths  $\geq\frac{\ln n}{\mu}+\ln^{0.5+\epsilon} n$ or depths $\leq\frac{\ln n}{\mu}+\ln^{0.5+\epsilon} n$ for any choice of $\epsilon>0$. We also find the first asymptotic of the variances of the depths of the balls in the tree.
}\end{abstract}

\section{Introduction}
\subsection {Preliminaries}

In this paper we consider random split trees introduced by Devroye \cite{devroye3}. Some  important examples of split trees are binary search trees, $m$-ary search trees, quadtrees, median of $(2k+1)$-trees, simplex trees, tries and digital search trees. As shown in \cite{devroye3} the split trees belong to the family of so-called 
$\log{n}$ trees, i.e., trees with height (maximal depth) $a.a.s.$\ $\mathcal O (\log{n})$. (For the notation $a.a.s$, see \cite{jan4}.)


The (random) split trees constitute a large class of random trees which are recursively generated. Their formal definition is given in the ``split tree generating algorithm'' below. To facilitate the penetration of this rather complex algorithm we will first provide a brief heuristic description. 

A skeleton tree $S_b$ of  branch factor $b$ is an infinite rooted tree in which each vertex
 has exactly $b$ children that are numbered $1,2,\dots ,b$.
A split tree is a finite subtree of a skeleton tree $S_b$. The split tree is constructed recursively by distributing balls one at a time to a subset of vertices of $S_b$.  We say that the tree has cardinality $n$ if $n$ balls are distributed. Since many of the common split trees come from algorithms in Computer Science the balls often represent some ``key numbers'' or other data symbols. There is also a so-called vertex capacity, $s>0$, which means that each node can hold at most $s$ balls.
We say that a vertex $v$ is a leaf in a split tree if the node itself holds at least one ball but no descendants of $v$ hold any balls. The split tree consists of the leaves and all the ancestors of the leaves, in particular the root of $S_b$, but no descendant of a leaf is included. In this way the definition of leaves in split trees is equivalent to the usual definition of leaves in trees. See Figure \ref{boll} and Figure \ref{boll2}, where two examples of split trees are illustrated (the parameters $s_0$ and $s_1$ in the figures are introduced in the formal ``split tree generating algorithm'').

The first ball is placed in the root of $S_b$.
  A new ball is added to the tree by starting at the root, and then letting the ball fall down to lower levels in the tree until it reaches a leaf. 
Each vertex $v$ of $S_b$ is given an independent copy of the so-called random split vector $\mathcal{V}=(V_1,V_2 \dots ,V_b)$
 of probabilities, where $\sum_i V_i=1$ and $V_i\geq 0$.
The split vectors control the path that the ball takes until it finally reaches a leaf; when the ball falls down one level from vertex $v$ to one of its children, it chooses the $i$-th child  of $v$  with probability $V_i$, 
i.e., the $i$-th component of the split vector associated to $v$.
When a full leaf (i.e., a leaf which already holds $s$ balls) is reached by a new ball it splits. This means that some of the $s+1$ balls are given to its children, leading to new leaves so that more nodes will be included in the tree. 
When all the $n$ balls are distributed we end up with a split tree with a finite number of nodes which we denote by the parameter $N$.

\textbf{The split tree generating algorithm:} The formal, comprehensive ``split tree generating algorithm'' is as follows with the following introductory notation.
The (random) split tree has the parameters $b,n,s$ and $\mathcal{V}$ as we described above; there are also two other parameters: $s_0,s_1$ (related to the parameter $s$) that occur in the algorithm below.
 Let $n_v$ denote the total number of balls that the vertices 
in the subtree rooted at vertex $v$ hold together, and $C_v$ be the number of balls that are held by $v$ itself. 
Thus, we note that a vertex $v$ is a leaf if and only if $C_v=n_v>0$. Also note that a vertex $v\in S_b$ is included in the split tree if, and only if, $n_v>0$. If  $n_v=0$, the vertex $v$ is not included and it is called useless. 

Below there is a description of the algorithm which determines how the $n$ balls are distributed over the vertices. 
Initially there are no balls, i.e., $C_v=0$ for each vertex $v$. 
Choose an independent copy  $\mathcal{V}_v$ of  $\mathcal{V}$ for every vertex $v \in S_b$. Add balls one by one to the root by the following recursive procedure for adding a ball to the subtree rooted at $v$.

\begin{enumerate}

\item\label{first} If $v$ is not a leaf, choose child $i$ with probability $V_i$,
and recursively add the ball to the subtree rooted at child $i$, by the rules given in steps \ref{first}, \ref{second} and \ref{last}.

\item\label{second}  If $v$ is a leaf and $C_v=n_v<s$, ($s$ is the capacity of the vertex) then add the ball to $v$ and stop. 
Thus, $C_v$ and $n_v$ increase by 1.

\item \label{last} If $v$ is a leaf and  $C_v=n_v=s$, the ball cannot be placed at $v$ since it 
is occupied by the maximal number of balls it can hold. In this case let $n_v=s+1$ and $C_v=s_0$, by placing $s_0\leq s$ 
randomly chosen balls at $v$ and $s+1-s_0$ balls at its children. This is done by first giving $s_1$ randomly chosen balls
to each of the $b$ children. The remaining $s+1-s_0 -bs_1$ balls are placed by choosing a child for each ball independently 
according to the probability vector $\mathcal{V}_v=(V_1,V_2, \dots ,V_b)$, and then using the algorithm described in steps 
\ref{first}, \ref{second} and \ref{last} applied to the subtree rooted at the selected child. Note that if $s_0>0$ or $s_1>0$, this procedure does not need to be repeated since no child could reach the capacity $s$, 
whereas in the case 
$s_0=0$ this procedure may have to be repeated several times.

\end{enumerate}

From \ref{last}. it follows that the integers  $s_0$ and $s_1$  have to satisfy the inequality 
\begin{align*}0\leq s_0\leq s,~0\leq bs_1\leq s+1-s_0.
\end{align*}
Note that every nonleaf vertex has $C_v=s_0$ balls and every leaf has $0<C_v\leq s$ balls.

\begin{figure}[h]
\begin{center}
\includegraphics[scale=0.40]{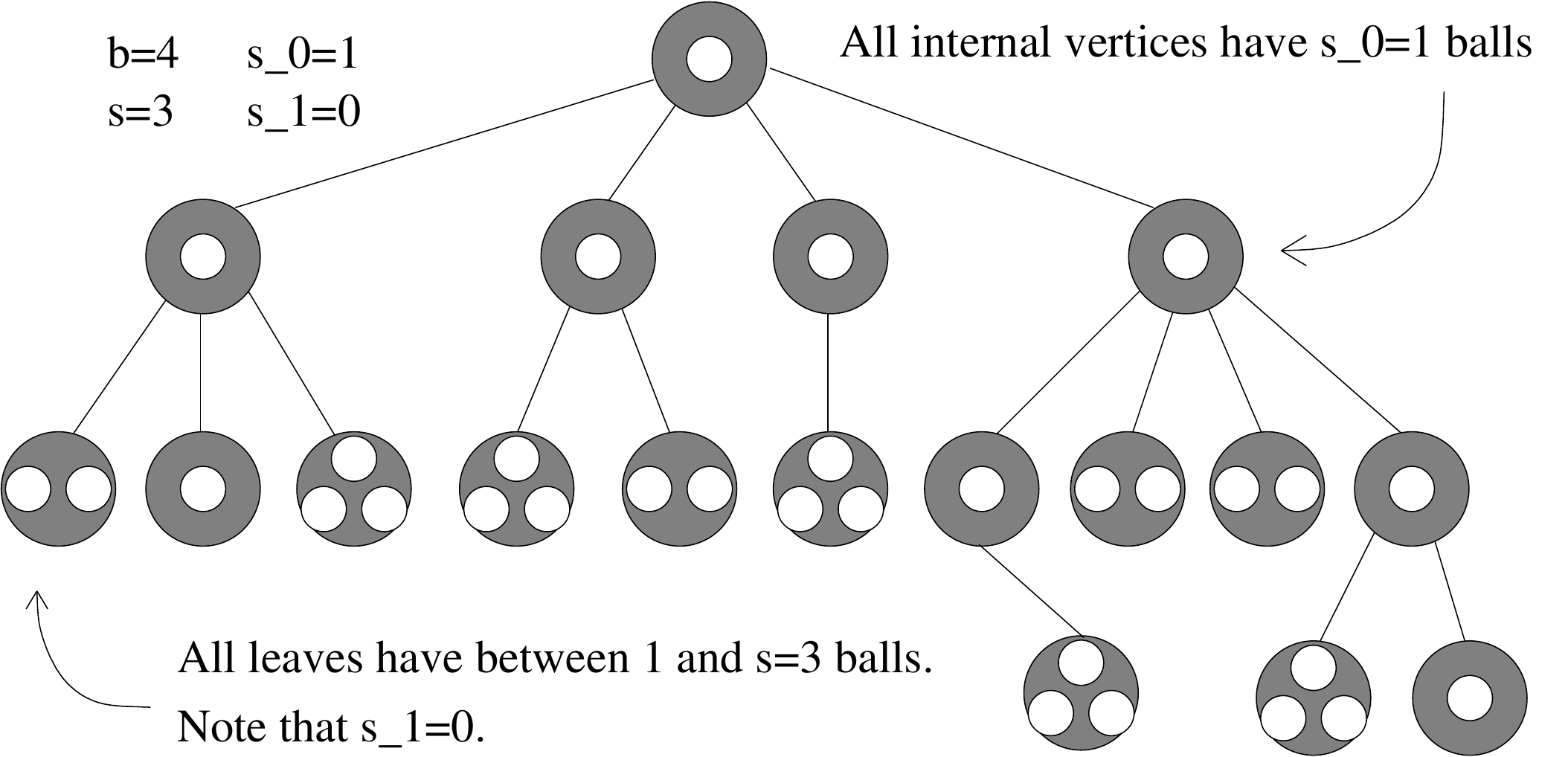}
\end{center}
\caption{\small{\textit{This figure illustrates a split tree with parameters
$b=4,~s=3,~s_0=1$ and $s_1=0$. }}}
\label{boll}
\end{figure}

\begin{figure}[h]
\begin{center}
\includegraphics[scale=0.40]{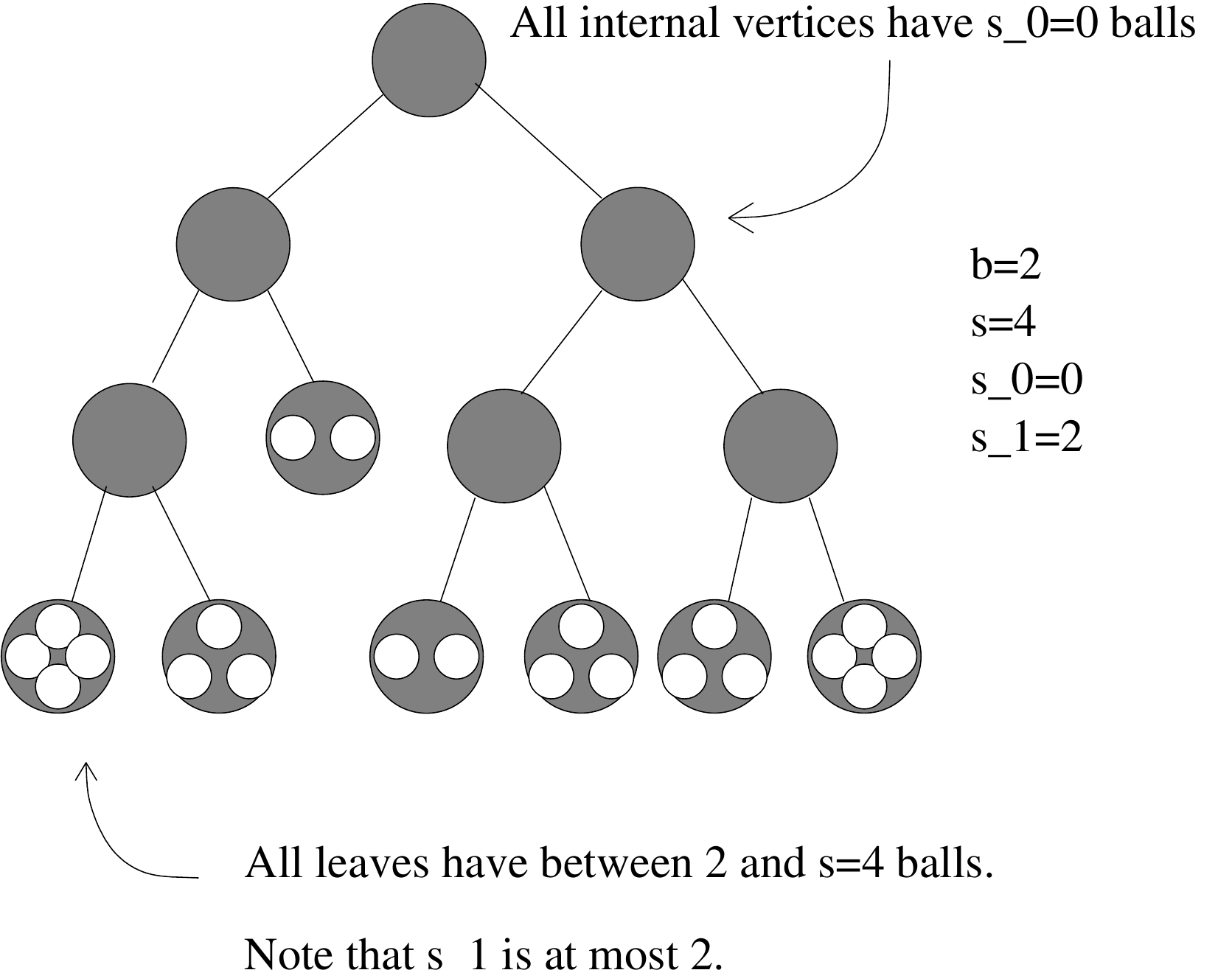}
\end{center}
\caption{\small{\textit{ This figure illustrates a split tree with parameters
$b=2,~s=4,~s_0=0$ and $s_1=2$.}}}
\label{boll2}
\end{figure}

 Figure \ref{boll} shows a split tree  with cardinality 32 and parameters $(b,s,s_0,s_1)=(4,3,1,0)$ and Figure \ref{boll2} shows a split tree with cardinality 21 and parameters $(b,s,s_0,s_1)=(2,4,0,2)$.

We can assume that the components $V_i$ of the split vector $\mathcal{V}$ are identically 
distributed. If this were not the case they can anyway be made identically distributed by using a random permutation, see \cite{devroye3}. Let $V$ be a random variable with this distribution.
This gives (because $\sum_i V_i=1$) that $\mathbf{E}(V)=\frac{1}{b}$.
We use the notation $T^{n}$ to denote a split tree with $n$ balls. However, note that even conditioned on the fact that the split tree has $n$ balls, the number of nodes $N$, is still a random number.
The only parameters that are important in this work (and in general these parameters are the important ones for most results concerning split trees) are the cardinality $n$, the branch factor $b$ and the split vector $\mathcal{V}$; this is illustrated in Section \ref{Subtrees}.
As an example, in the binary search tree considered as a split tree, $b=2$ and the split vector $\mathcal{V}$ is $(U,1-U)$ where $U$ is a uniform $U(0,1)$ random variable. This is a beta $(1,1)$ random variable. In fact for many important split trees $V$ is beta-distributed.
(The other parameters for the binary search tree considered as a split tree are $s=1,~s_0=1$ and $s_1=0$.) For the binary search tree the number of balls $n$ is the same as the number of vertices $N$; this is not true for split trees in general.

\subsection{Notation}\label{not}
In this section some of the notation that we use in the present study is collected. 

Let $T^{n}$ denote a split tree with $n$ balls; for simplicity we often write $T$. Let $\mathsf{V}\big(T\big)$ denote the set of vertices in a rooted tree $T$. We write $|S|$ for the number of vertices in a set $S$. Note that for the number of vertices $N$ we have  $N=\big|\mathsf{V}\big(T^{n}\big)\big|$. 
Let $T_v$ be a subtree rooted at $v$.
Let $n_v$ denote the number of balls in the subtree rooted at vertex $v$ and let $N_v$ denote the number of vertices. Note that $N_v=\big|\mathsf{V}\big(T_v\big)\big|$.

Let $D_n$ denote the depth of the last inserted ball in the tree and $D_k$ the depth of the $k$-th inserted ball when all $n$ balls have been added. Let $D_n^{*}$ be the average depth, i.e., $D_n^{*}=\dfrac{\sum_{k=1}^{n}D_k}{n}$. We also use the notation $D_{k}^{f}$ for the depth of the node of ball $k$  when it is added to the tree; this could differ from $D_k$ since the ball can move during the splitting process: $D_{k}^{f}\leq D_k$. Equivalently, $D_{k}^{f}$ is the depth of the last ball in a split tree with $k$ balls. 
Let $d(v)$ denote the depth (or height) of a vertex $v$, sometimes we just write $d$ for the depth of $v$. 

Let $p(v)$ denote the parent of a vertex $v$.

 There are at least two different types of total path lengths in a tree $T$ that are of interest: the sum of all depths (distances to the root) of the balls in $T$, and the sum of all the depths of the vertices in $T$. 
We denote the former by $\Psi(T)$  and the latter by $\Upsilon(T)$.

We use the standards notations, $N(\mu,\sigma^2)$ for a normal distribution with expected value $\mu$ and variance $\sigma^2$, and Bin($m,p$) for a random variable $X$ with a binomial distribution with parameters $m$ and $p$. We also use the notation mixed binomial distribution $(X,Y)$ or for short mBin$(X,Y)$ for a binomial distribution where at least one of the parameters $X$ and  $Y$ is a random variable (the other one could be deterministic).   Let $H_n$ denote the height of a split tree with $n$ balls.

 Let
$T_{v_i},~i\in \{1,\dots,b^L\}$ be the subtrees rooted at 
depth $L=\lfloor \beta \log_b\ln{n}\rfloor$ for some constant $\beta$. For simplicity we just write $T_{i},~i\in \{1,\dots,b^L\}$ for these.
Let $d_w(v):=d(v)-d(w)$, i.e., the depth of a vertex $v$ in the subtree $T_w$. In particular we write $d_i(v):=d(v)-L$ for the depth of a vertex $v$ in the subtrees $T_i,~i\in\{1,\ldots,b^L\}$.

Recall that $V$ is a random variable with the distribution of the identically distributed components $V_i,~i\in \{1,\dots,b\}$ in the split vector $\mathcal{V}=(V_1,\dots,V_b)$. 
Let $\Delta=V_S$ be the size biased distribution of $(V_1,\dots,V_b)$, i.e., given $(V_1,\dots,V_b)$, let $\Delta=V_j$ with probability $V_j$, see \cite{devroye3}.
Let, \begin{align*} c:=\mathbf{E}(\Delta)=b\mathbf{E}(V^2),
          \end{align*} 
and
\begin{align}\label{splitdef}\mu &:=\mathbf{E}\Big(-\ln \Delta\Big)=b\mathbf{E}\Big(-V\ln{V}\Big),\nonumber\\
\sigma ^2&:=\mathbf{Var}\Big(\ln \Delta\Big)=b\mathbf{E}\Big(V\ln^2 V\Big)-\mu ^2.
\end{align} Note that the second equalities of $\mu$ and $\sigma$ imply that they are bounded. Similarly all moments of $-\ln \Delta$ are bounded.

For a given $\epsilon>0$,  we say that a vertex $v$ in $T^{n}$  is ``good''  if 
\begin{align}\label{good1}\mu^{-1}\ln{n}-\ln^{0.5+\epsilon}{n}\leq d(v)\leq \mu^{-1}\ln{n}+\ln^{0.5+\epsilon}{n},
\end{align}
and ``bad'' otherwise. We write $\mathsf{V}^{\ast}\big(T^{n}\big)$ for the set of good vertices in $T^n$, and for the number of good vertices we write $N^{*}:=|\mathsf{V}^{\ast}\big(T^{n}\big)|$.

We say that $Y_m=o_p(a_m)$ if $a_m$ is a positive number and 
$Y_m$ is a random variable such that $Y_m/a_m \stackrel{p}{\rightarrow}0$ as $m\rightarrow \infty$. 
We use two unusual types of order notation; let $a_m$ be  a positive number and $Y_m$  a random variable, by the notation $Y_m:=\mathcal O _{L^p}(a_m)$ we mean that 
 $(\mathbf{E}({Y_m}^p))^{\frac{1}{p}}\leq Ca_m$ for some constant $C$, and by the notation $Y_m:=o _{L^p}(a_m)$ we mean that
$(\mathbf{E}({Y_m}^p))^{\frac{1}{p}}/a_m\rightarrow0$.
We use the notation $\Omega_j$ for the $\sigma$-field generated by $\{n_v,~d(v)\leq j\}$.
Finally we write  $\mathscr{G}_j$ for
 the $\sigma$-field generated by the $\mathcal{V}$-vectors for all vertices $v$ with $d(v)\leq j$. 

 \subsection{A weak law and a central limit law for the depth}\label{strong}

 In \cite{devroye3} Devroye presents a weak law of large numbers and a central limit law for $D_n$ (the depth of the last inserted ball). 
If $\mathbf{P}(V=1)=0$ and $\mathbf{P}(V=0)<1$ then
\begin{align}\label{splitdepth}\frac{D_n}{\ln n}\stackrel{p}\rightarrow \mu^{-1},
\end{align}and
\begin{align}\label{expsplitdepth}\frac{\mathbf{E}(D_n)}{\ln n}\rightarrow \mu^{-1}.
\end{align}
From the following lemma it follows easily (as we explain below) that (\ref{expsplitdepth}) also holds for the average depth $D_{n}^{*}$. Recall that $D_k$ is the depth of the $k$-th ball in the tree when all $n$ balls are added.
\begin{Lemma}\label{comparingdepths}For $i \leq j$, we have $D_{i}\leq D_{j}$ in stochastic sense.
\end{Lemma}

\begin{proof}

 We show this by showing that for an arbitrary $i\in \{1,\dots,n-1\}$, $D_i\leq D_{i+1}$, where the inequalities and equalities below are in the stochastic sense only.
 We show this by the use of coupling arguments.  

First consider two identical copies $T$ and $\widehat{T}$ of the split tree when $i-1$ balls have been added, where we let $\widehat{v}$ in $\widehat{T}$ denote the corresponding vertex of $v$ in $T$. More precisely, we consider two split trees $T$ and $\widehat{T}$ with the same split vectors in all vertices of the infinite  skeleton tree, and if a ball $k$, $k\leq i-1$, is added to $v$ in $T$ then ball $k$ is added to  $\widehat{v}$ in $\widehat{T}$.  We now assume that we add the two balls $i$ and $i+1$ to $T$ and $\widehat{T}$. 

If ball $i$ and ball $i+1$ are added to different leaves $l_1$ and $l_2$ in $T$ then in $\widehat{T}$ we let them switch positions, i.e., ball $i$ is added to $\widehat{l}_2$ and ball $i+1$ is added to $\widehat{l}_1$. (Recall the notation $D_{k}^{f}$ from Section \ref{not}.) Hence, it is obvious for reasons of symmetry that $D_{i}^{f}\stackrel{d}{=}D_{i+1}^{f}$. When the balls $\in\{i+2,\dots,n\}$ are added, we add them to the corresponding vertices in $T$ and $\widehat{T}$. Thus, the two trees are identical in the whole process except for that ball $i$ and ball $i+1$  always have switched positions in $T$ and $\widehat{T}$. Hence, by symmetry  $D_i\stackrel{d}{=}D_{i+1}$.

If ball $i$ and ball $i+1$ are added to the same leaf $l$ in $T$ then there are three different cases:

If $n_l\leq s-2$,  so that $l$ does not split when also ball $i$ and ball $i+1$ have been added, then $T$ and $\widehat{T}$ are still identical since ball $i$ and ball $i+1$ stay in $l$. When more balls are added we can again assume that ball $i$ and ball $i+1$ have switched positions in  $T$ and $\widehat{T}$ at every step of the the recursive construction until all $n$ balls are added. Hence, by symmetry $D_i\stackrel{d}{=}D_{i+1}$.

If $n_l=s-1$, so that $l$ gets $s+1$ balls when the new balls are added, $l$ splits according to the usual splitting process when ball $i+1$ is added. Again we let ball $i$ and ball $i+1$ switch positions in $T$ and $\widehat{T}$. This means that if ball $i$ is added to $v_1$ and ball $i+1$ is added to $v_2$ in ${T}$, then in $\widehat{T}$  ball $i$ is added to $\widehat{v}_2$ and ball $i+1$ is added to $\widehat{v}_1$. Thus, again by symmetry  $D_{i}^{f}\stackrel{d}{=}D_{i+1}^{f}$. By using the same type of argument as in the cases above we get $D_i\stackrel{d}{=}D_{i+1}$.

If $n_l=s$, so that $l$ in $T$ gets $s+2$ balls when the new balls are added, we let $l$ split according to the usual splitting process where $l$ keeps $s_0$ balls and sends the other balls to its children. 

If  ball $i$ is one of the $s_0$  balls in the children then it is obvious without using the coupling that $D_{i}^{f}\leq D_{i+1}^{f}$ and also $D_i\leq D_{i+1}$.

If  ball $i$ is not one of the $s_1$ balls in the children of $l$ in $T$ and ball $i$ is added to $v_1$ and ball $i+1$ is added to $v_2$, then in $\widehat{T}$ we can again assume that ball $i$ is added to $\widehat{v}_2$ and ball $i+1$ is added to $\widehat{v}_1$.
 Thus, in the stochastic sense $D_{i}^{f}\stackrel{d}{=} D_{i+1}^{f}$, and  $D_i\stackrel{d}{=} D_{i+1}$.

If ball $i$ is one of the $s_1$ balls in $T$, we use a related but not an identical type of coupling argument as in the previous cases. In this case ball $i$ is added by uniformly choosing one of the $b$ children of $l$ each with probability $\frac{1}{b}$, while ball $i+1$ is added by using the probabilities given by the components in the split vector $\mathcal{V}_l$ of $l$. Again $T$ and $\widehat{T}$ are identical until $i-1$ balls are added baring the possibility of variation in the split vectors of the vertices above the leaves as described below. 
If ball $i$ in $T$ goes to a child $v_1$ of $l$ related to a component $V_j$ in $\mathcal{V}_l$, then we add ball $i+1$ in $\widehat{T}$  to $\widehat{v}_1$ with probability $\min \{1,\frac{V_j}{1/b}\}$ and to one of the other children related to a component $V_k>1/b$ with probability $\max \{0,1-\frac{V_j}{1/b}\}$, so that the sum of the probabilities gives the right marginal distribution. Assume that ball $i$ is added to the child $v$ of $l$ in $T$ and  ball $i+1$ is added to the child  $\widehat{w}$ of $\widehat{l}$ in $\widehat{T}$.
This means that $\widehat{w}$  relates to a component of the split vector of $\widehat{l}$  at least as large as the component of the split vector of $l$ related to $v$.
Now we can assume that the split vectors in the vertices in the subtree rooted at $v$   correspond to the split vectors in the vertices in the subtree rooted at  $\widehat{w}$.
This means that we can assume that when ball number $j$ in the subtrees is added it goes to the corresponding vertex in both of the subtrees. However, note that the balls could have different labels if we consider their original label in the whole tree, since the subtree rooted in $\widehat{w}$ could have more balls than the subtree rooted in ${v}$. Thus,
as long as the subtrees have the same number of balls, new balls are added to the corresponding positions in these subtrees, and ball $i$ and ball $i+1$ are also held by vertices of corresponding positions. This construction shows that if the subtrees rooted in $\widehat{w}$ and  $v$ have $k$ and $l$ balls, respectively, where 
$k>l$, and ball $i$ in $T_{v}$ is in vertex $h$, then ball $i+1$ in $\widehat{T}$ is in a subtree of $\widehat{T}_{\widehat{w}}$ with root
corresponding to the position of $h$. This shows that in the stochastic sense $D_i\leq D_{i+1}$.

Hence, in all cases,  $D_{i}\leq D_{i+1}$ stochastically and thus for $i<j$, it follows that $D_{i}\leq D_{j}$ in stochastic sense. 
\end{proof}

This means in particular that for all  $k\leq n$, $\mathbf{E}(D_k)\leq \mathbf{E}(D_{n})$. Since the sum of $\frac{\mathbf{E}(D_k)}{n}$ for $k< \frac{n}{\ln^2 n}$ is $o(1)$, we can ignore the balls  $k< \frac{n}{\ln^2 n}$. 
We consider the balls $k\geq \frac{n}{\ln^2 n}$.  For ball $k\geq\frac{n}{\ln^2 n}$ it follows from (\ref{expsplitdepth}) that $\mathbf{E}(D_{k}^{f})\sim  \mu^{-1} \ln n$. Thus, since for all 
$k\geq \frac{n}{\ln^2 n}$, $\mathbf{E}(D_{k}^{f})\leq \mathbf{E}(D_k)\leq \mathbf{E}(D_{n})$, we get 
\begin{align}\label{expsplitdepth1}\frac{\mathbf{E}(D_n^{*})}{\ln n}\rightarrow \mu^{-1}.
\end{align}

Furthermore, see \cite[Theorem 1]{devroye3}, if $\sigma >0$,  and assuming that $V$ is not monoatomic, i.e., we don't have $\mathcal{V}\equiv \frac{1}{b}$,  
\begin{align}\label{splitdepth:2}\frac{D_n-\mu^{-1} \ln n}{\sqrt{\sigma ^2\mu^{-3}\ln n}}\stackrel{d}\rightarrow N(0,1),
\end{align}
where $N(0,1)$ denotes the standard Normal distribution and $\stackrel{d}\rightarrow$ denotes convergence in distribution. Tries are special forms of split trees with a random permutation of deterministic components $(p_1,p_2,\dots,p_b)$ and therefore not as random as many other examples. (In the literature tries have also been treated separately to other random trees of logarithmic height.) Of all the most common examples of split trees only some special cases of tries (the symmetric tries and symmetric digital search trees) have a monoatomic distribution of $V$.   From  (\ref{splitdepth:2}) it follows that ``most'' nodes lie at $\mu^{-1} \ln n+\mathcal{O}(\sqrt{\ln n})$. 

\subsection{Subtrees}\label{Subtrees}
For the split tree  where the number of balls $n>s$, there are $s_0$ balls in the root vertex and the cardinalities of the 
$b$ subtrees are distributed as $(s_1,\dots, s_1)$ plus a multinomial vector $(n-s_0-bs_1,V_1,\dots, V_b)$.
Thus, conditioning on the random $\mathcal{V}$ -vector that belongs to the root, the subtrees rooted at the children have cardinalities close to $nV_1,\dots, nV_b$.
This is often used in applications of random binary search trees. In particular, we used this fact frequently in \cite{holmgren}.
 ``The split tree generating algorithm'' described above, and the fact that a mBin$(X,p_1)$ in which $X$ is  Bin$(m,p_2)$ is distributed as a  Bin$(m,p_1p_2)$, give in a stochastic sense, an upper bound on the number of balls $n_v$ in a subtree rooted at a vertex $v$:
Let $v$ be a vertex at depth $d$, conditioning on $\mathscr{G}_d$  (i.e., the $\sigma$-field generated by the $\mathcal{V}$ vectors for all vertices $v$ with $d(v)\leq d$), gives
\begin{align}\label{binomial} n_v &\leq \mathrm{Bin}(n,\prod_{j=1}^{d}W_{j,v})+\mathrm{Bin}(s_1,\prod_{j=2}^{d}W_{j,v})\nonumber\\&+\mathrm{Bin}(s_1,\prod_{j=3}^{d}W_{j,v})+\dots + \mathrm{Bin}(s_1,W_{d,v})+s_1,
\end{align}
where $W_{j,v},j\in\{1,\dots d\}$ are i.i.d.\ random variables given by the split vectors associated with the nodes in the unique path from $v$ to the root. This means in particular that $W_{j,v}\stackrel{d}{=}V$. However, we note that the terms in (\ref{binomial}) are not independent. Also observe that $\mathscr{G}_d$ is equivalently the $\sigma$-field generated by $W_{j,v},j\in\{1,\dots ,d\}$ for all $v$ with $d(v)=d$. Similarly, we also have a lower bound for $n_v$, i.e., for $v$ at depth $d$, conditioning on $\mathscr{G}_d$ in stochastic sense,
\begin{align}\label{binomial2} n_v &\geq \mathrm{Bin}(n,\prod_{j=1}^{d}W_{j,v})-\mathrm{Bin}(s,\prod_{j=2}^{d}W_{j,v})\nonumber\\&-\mathrm{Bin}(s,\prod_{j=3}^{d}W_{j,v})-\dots - \mathrm{Bin}(s,W_{d,v});
\end{align}
we can replace the term $s$ by $s_0+bs_1\leq s$ for a sharper bound. As in (\ref{binomial}) the terms in (\ref{binomial2}) are not independent.

Recall that for a Bin$(m,p)$ distribution, the expected value is $mp$ and the variance is $mp(1-p)$.
Thus,  Chebyshev's inequality applied to the dominating term  Bin$(n,\prod_{j=1}^{d}W_{j,v})$ in (\ref{binomial}) gives that $n_v$ for $v$ at depth $d$ is close to
\begin{align}\label{sub} M_{v}^{n}:=nW_{1,v}W_{2,v} \dots W_{d,v}.
\end{align} 
More precisely by using (\ref{binomial}) and (\ref{binomial2}), the Chebyshev and Markov inequalities give for $v$ with $d(v)=d$, that for large $n$,
\begin{align}\label{iliada1,1}&\mathbf{P}\big(\mid n_v-n\prod_{j=1}^{d}W_{j,v}\mid > n^{0.6}\big)\leq 4\dfrac{\mathbf{E}\Big(\mathbf{Var}\Big(\mathrm{Bin}(n,\prod_{j=1}^{d}W_{j,v})\Big|\mathscr{G}_d\Big)\Big)}{n^{1.2}}\nonumber\\+&
4\dfrac{\mathbf{E}\Big(\mathbf{E}\Big(\mathrm{Bin}(s,\prod_{j=2}^{d}W_{j,v})+\mathrm{Bin}(s,\prod_{j=3}^{d}W_{j,v})+\dots +s\Big|\mathscr{G}_d\Big)\Big)}{n^{0.6}}\nonumber\\&~~~~~~~~~~~~~~~~~~~~~~~~~~~~~~~~~~~~~~~~\leq\frac{4nb^{-d}}{n^{1.2}}+\frac{\sum_{k=1}^{\infty}4sb^{-k}}{n^{0.6}}
\leq\frac{1}{n^{0.1}}.\end{align}
Since the $n_v$'s (conditioned on the split vectors) for all $v$ at the same depth are identically distributed, we sometimes skip the vertex index of $W_{j,v}$ and just write $W_j$.

\subsection{Renewal Theory}
Renewal theory is a widely used branch of probability theory that generalizes Poisson processes to arbitrary holding times.
A classic in this field is Feller \cite{feller} on recurrent events.
First we recollect some standard notation.
Let $X_0=0$ a.s.. Let $X_k,~k\geq 1$, be i.i.d.\ nonnegative random variables distributed as $X$ and let $S_m,~m\geq 1$, be the partial sums. Let $F$ denote the distribution function of $X$, and let $F_m$ be the distribution function of $S_m,~m\geq0$. Thus, for $x\geq0$,
\begin{align*} F_0(x)=1,~~F_1(x)=F(x),~~
F_m(x)=F^{m^\ast}(x),
 \end{align*}
i.e., $F_m$ equals the $m$-fold convolution of $F$ itself.
The ``renewal counting process''  $\{\mathcal{N}(t),~t\geq 0\}$ is defined by 
\begin{align*} \mathcal{N}(t):=\max \{m:S_m\leq t\},
\end{align*} which one can think of as the number of renewals before time $t$ of an object with a lifetime distributed as the random variable $X$. In the specific case when $X\stackrel{d}= \mathrm{Exp}(\lambda)$, 
$\{\mathcal{N}(t),~t\geq 0\}$ is a ``Poisson process''.
An important well studied function is the so called ``standard renewal function'' defined as 
\begin{align}\label{stand}V(t):=\sum_{m=0}^{\infty}F_m(t),
 \end{align}
 which one can easily show is equal to $\mathbf{E}(\mathcal{N}(t))$.
The renewal function $V(t)$ satisfies the so called renewal equation 
\begin{align*}V(t)=1+(V\ast dF)(t),~~t\geq 0.
 \end{align*}
For a broader introduction to renewal theory, see e.g.\  \cite{asmussen}, \cite{feller1}, \cite{feller2} and \cite{gut2}.
One of the main purposes of this study is to introduce renewal theory in the context of split trees.
Recall from (\ref{sub}) in Section \ref{Subtrees}
that the subtree size $n_v$ for $v$ at depth $k$, is close to
$nW_{1}W_{2}\dots W_{k}$, where $W_{j},~j\in\{1,\dots ,k\}$, are independent random variables distributed as $V$. 
Now let $Y_{k}:=-\sum_{j=1}^{k}\ln {W_{j}}$, and for simplicity we also denote the summands $\varpi_r:=-\ln {W_{j}}$. Note that $nW_{1}W_{2}\dots W_{k}=ne^{-Y_{k}}$.
Recall that in a binary search tree, the split vector $\mathcal{V}=(V_1,V_2)$ is distributed as $(U,1-U)$ where $U$ is a uniform $U(0,1)$ random variable. For this specific case of a split tree the sum $Y_{k}$, (where $W_{j},~j\in\{1,\dots ,k\}$, in this case are i.i.d.\ uniform $U(0,1)$ random variables) is distributed as a $\Gamma(k,1)$ random variable.
This fact is used by, e.g., Devroye in \cite{devroye2} to determine the height of a binary search tree.
For general split trees there is no simple common distribution function of $\sum_{j=1}^{k} \ln W_j$, instead  renewal theory can be used. 

Let \begin{align*}\nu_k(t):=b^k\mathbf{P}(Y_k\leq t).
\end{align*}
We define the renewal function 
\begin{align}\label{renewal function}
 U(t):=\sum_{k=1}^{\infty}\nu_k(t).
\end{align}
We also denote $\nu(t):=\nu_1(t)=b\mathbf{P}(\varpi_r\leq t)$.
For $U(t)$ we obtain the following renewal equation
\begin{align}\label{renewal equation}
 U(t)&=\nu(t) +\sum_{k=1}^{\infty}(\nu_k\ast d\nu)(t)=\nu(t)+(U\ast d\nu)(t).
\end{align}

\section{Main Results}

In this section we present the main theorems of this work.

\begin{ass}\label{assumption2}   
In this work we assume as in Section \ref{strong} that $\mathbf{P}(V=1)=0$, and we also assume for simplicity  that $\mathbf{P}(V=0)=0$ and that $-\ln {V}$ is non-lattice. 
 \end{ass}
The reason for the non-lattice assumption (A1) is that we use renewal theory and there it often becomes necessary to distinguish between lattice and non-lattice distributions. Note that the assumption that $V$ is not monoatomic in Section \ref{strong} is included in the assumption that $-\ln {V}$ is non-lattice. Again of the common split trees only for some special cases of tries and digital search trees does $-\ln {V}$ have a lattice distribution.
Our first main result is on the relation between the number of vertices $N$ (recall that this is a random variable) and the number of balls $n$.
\begin{thm}\label{assumption1}
There is a constant $\alpha$ depending on the type of split tree such that 
\begin{align}\label{vertexball}\mathbf{E}(N)=\alpha n+o\big(n\big),
\end{align}
 and 
\begin{align}\label{vertexball2}\mathbf{Var}(N)=o\big(n^{2}\big).
\end{align}
 \end{thm}

Recall that there is a central limit law for the depth $D_n$ in (\ref{splitdepth:2}) so that most vertices are close to $\frac{\ln n}{\mu}+\mathcal{O}\Big(\sqrt{\ln n}\Big)$, our next result sharpens this fact.
Recall that for any constant $\epsilon>0$, we say that a vertex $v$ in $T^{n}$  is ``good''  if 
\begin{align*}\mu^{-1}\ln{n}-\ln^{0.5+\epsilon}{n}\leq d(v)\leq \mu^{-1}\ln{n}+\ln^{0.5+\epsilon}{n},
\end{align*} and ``bad'' otherwise. 
\begin{thm}\label{lemma1} 
For any choice of $\epsilon>0$, the number of bad nodes in $T^{n}$  is bounded by $\mathcal O_{L^{1}} \Big(\frac {n}{\ln^{k}{n}}\Big)$ for any constant $k$.
\end{thm} 

In the third main result we sharpen the limit laws in (\ref{expsplitdepth}) and (\ref{expsplitdepth1}) for the expected value of the depth of the last ball $D_n$ and the average depth $D_n^{*}$. We also find the first asymptotic of the variances of the $k$:th ball $D_k$ for all $k$, $\frac{n}{\ln n}\leq k \leq  n$.
\begin{thm}\label{lemma} 
For the expected value of the depth of the last ball we have 
\begin{align}\label{expsplitdepth2} \dfrac{\mathbf{E}(D_n)-\mu^{-1}\ln n}{\sqrt{\ln n}}\longrightarrow 0,
 \end{align} 
and the same result holds for the average depth $D_n^{*}$, i.e., 
\begin{align}\label{expsplitdepth3} \dfrac{\mathbf{E}(D_n^{*})-\mu^{-1}\ln n}{\sqrt{\ln n}}\longrightarrow 0.
 \end{align} 
Furthermore, for the variance of the depth of the $k$:th ball we have that for all $\frac{n}{\ln n}\leq k \leq  n$,  
\begin{align}\label{varsplitdepth2} \dfrac{\mathbf{Var}(D_k)}{\ln n}\longrightarrow \sigma^2\mu^{-3}.
 \end{align}
\end{thm}

We complete this section by stating two corollaries of Theorem \ref{lemma}. Recall that we write $\mathsf{V}^{\ast}\big(T^{n}\big)$ for the set of good vertices in $T^{n}$, i.e., those with depths that belong to the strip in (\ref{good1}).
\begin{cor}\label{cor1} Summing over all vertices give 
 \begin{align}\label{branoder22}
 \mathbf{E}\Big(\sum_{v\in \mathsf{V}\big(T^{n}\big)}(d(v)-\mu^{-1}\ln{n})^2\Big)=\alpha n \mu^{-3}\sigma^2\ln n+o(n\ln n).
\end{align}
For the good vertices we also have 
 \begin{align}\label{branoder2}
 \mathbf{E}\Big(\sum_{v\in \mathsf{V}^{\ast}\big(T^{n}\big)}(d(v)-\mu^{-1}\ln{n})^2\Big)=\alpha n \mu^{-3}\sigma^2\ln n+o(n\ln n).
\end{align}
\end{cor}

We write $\mathsf{V}^{\ast}\big(T_i\big)$ for the set of good vertices in $T_i$.
\begin{cor} \label{cor2} Let  $L=\lfloor \beta \log_b\ln{n}\rfloor$ for some large constant $\beta$. Then, summing over all vertices give
\begin{align}\label{indsum33}&\sum_{i=1}^{b^L}\sum_{v\in \mathsf{V}\big(T_i\big)}\frac {{(d_i(v)-\mu^{-1}\ln{n_i})^2}}{ \mu^{-3}\ln^3{n_i}}=\frac{\sigma^2\alpha n}{\ln^2{n}}+o_p\Big(\frac {n}{{\ln^2{n}}}\Big)\end{align}
and for the good vertices we also have 
\begin{align}\label{indsum3}&\sum_{i=1}^{b^L}\sum_{v\in \mathsf{V}^{\ast}\big(T_i\big)}\frac {{(d_i(v)-\mu^{-1}\ln{n_i})^2}}{ \mu^{-3}\ln^3{n_i}}=\frac{\sigma^2\alpha n}{\ln^2{n}}+o_p\Big(\frac {n}{{\ln^2{n}}}\Big).\end{align}
\end{cor}

\section{Some Fundamental Renewal Theory Results} 
\label{mainsection}

The main goal of this section is to present a renewal theory lemma and a corollary of this lemma, which are both frequently used in this study. 
In contrast to standard renewal theory the distribution function $\nu(t)$ in (\ref{renewal equation}) is not a probability measure. 
However, to solve (\ref{renewal equation}) we can apply \cite[Theorem VI.5.1]{asmussen} which deals with non probability measures. The result we get is presented in the following lemma. 
\begin{Lemma}\label{lem5} The renewal function $U(t)$
in (\ref{renewal function}) satisfies  
\begin{align}\label{renewal equation3}
 U(t)=(\mu^{-1}+o(1))e^t~~as~t\rightarrow \infty.
\end{align}
\end{Lemma}

\begin{proof}
 Since the distribution function $\nu(t)$ is not a probability measure, we define another (``conjugate'' or ``tilted'') measure $\omega$ on $[0,\infty)$ by
\begin{align*}
d\omega(t)=e^{-t}d\nu(t).
\end{align*}
Recall from Section \ref{not} that $\Delta=V_S$ is the size biased distribution of 
$(V_1,\dots,V_b)$.
We note that $\omega(x)$ is the distribution function of the random variable 
$-\ln \Delta$ since
\begin{align*}
\mathbf{P}\Big(-\ln \Delta\leq x\Big)=\mathbf{E}\Big(\mathbf{E}\Big(I\{-\ln V_S\leq x\}\Big|(V_1,\dots,V_b)\Big)\Big)=\\
\mathbf{E}\Big(\sum_{i=1}^{b}I\{-\ln V_i\leq x\}V_i\Big)=
b\mathbf{E}\Big(I\{-\ln V\leq x\}e^{-\ln(V)}\Big)=\omega(x).
\end{align*}
Thus, $\omega$  is a probability measure.  
Further, by recalling $\mu:=\mathbf{E}(-\ln \Delta)$ and $\sigma^2:=\mathbf{Var}(-\ln \Delta)$ gives
\begin{align}\label{expren}
\mathbf{E}(\omega)=\mu,~~~\mathrm{and}~\mathbf{Var}(\omega)=\sigma^2.
\end{align}

Define $\widehat{U}(t):=e^{-t}U(t)$ and $\widehat{\nu}(t):=e^{-t}\nu(t)$. We shall apply \cite[Theorem VI.5.1]{asmussen}, but first we need to show that the condition  that $\widehat{\nu}(t)$ is ``directly Riemann integrable''  (d.R.i.) is satisfied.
 Note that $\widehat{\nu}(t)\leq b e^{-t}$, and thus since $\widehat{\nu}(t)$ is also continuous almost everywhere, by \cite[Proposition IV.4.1.(iv)]{asmussen} it follows that $\widehat{\nu}(t)$ is d.R.i. if $ b e^{-t}$ is d.R.i.. 
That $b e^{-t}$ is d.R.i.  follows by applying \cite[Proposition IV.4.1.(v)]{asmussen}, since $b e^{-t}$  is a  nonincreasing and Lebesgue integrable function. Then by applying \cite[Theorem VI.5.1]{asmussen} and (\ref{expren}) we get
 \begin{align}\label{ny} \widehat{U}(t)=\widehat{\nu}(t)+(\widehat{U}*d\omega)(t),
\end{align}
where $\omega(t)$ is a probability measure, and
\begin{align}\label{renewal equation4}\widehat{U}(t)\rightarrow \mu^{-1}\int_{0}^{\infty}\widehat{\nu}(x)dx=\mu^{-1}\int_{0}^{\infty}\nu(x)e^{-x}dx=:\kappa.
 \end{align}

 Integration by parts now gives \begin{align}\label{reneq}\kappa&=\mu^{-1}\Big(b\Big{|}-e^{-t}\mathbf{P}(\varpi_r\leq t)\Big{|}_{0}^{\infty}-\int_{0}^{\infty}-e^{-t}d\nu(t)\Big)\nonumber\\
&=\mu^{-1}b\mathbf{E}(e^{-\varpi_r})=\mu^{-1}.
\end{align}
Thus, $U(t)=(\mu^{-1}+o(1))e^t$.
\end{proof}

The following result is a very useful corollary of Lemma \ref{lem5}.
 We write for $v$ at depth $d(v)$, $M_{v}^{n}:=n\prod_{j=1}^{d(v)}W_j$. Recall from (\ref{sub}) in Section \ref{Subtrees} that this is close to the real subtree size $n_v$.
\begin{cor}\label{lem6,1} By taking the sum over vertices $v,~d(v)=k$ and letting $\frac{n}{K}\rightarrow \infty$, we get that the expected number of nodes with $M_{v}^{n}\geq K$ is equal to
\begin{align}\label{ej3}
\mathbf{E}(\Big|v\in V(T^n);~M_{v}^{n}\geq K\Big|)&=\sum_{k=0}^{\infty}b^k\mathbf{P}\Big(M_{v}^{n}\geq K\Big)\nonumber\\&=:U(\ln n-\ln K)+1= (\mu^{-1}+o(1))
\frac{n}{K}.
\end{align}\end{cor} 
\begin{proof}

 By using Lemma \ref{lem5} we get 
\begin{align}\label{sub55}  \sum_{d=0}^{\infty} b^d\mathbf{P}\Big(n\prod_{j=1}^{d}W_{j,v}\geq K\Big)&=
\sum_{d=0}^{\infty} b^d\mathbf{P}\Big(Y_d \leq \ln n-\ln K\Big)\nonumber\\&=(\mu^{-1}+o(1))
\frac{n}{K}.\end{align}
\end{proof}

We complete this section with a more general result in renewal theory, and a corollary of a more specific result that is valid for the renewal function $U(t)$ in (\ref{renewal function}). 
\begin{thm}\label{thm6} 
Let $F$ be a non-lattice probability measure and suppose that 
$0<\mu=\mathbf{E}(X)=\int_{0}^{\infty}xdF(x)<\infty$ and $\mathbf{E}(X^2)=\sigma^2+\mu^2<\infty$.

Let 
\begin{align}\label{standardrenewal}
 Z(t)=z(t)+\int_{0}^{t}Z(t-u)dF(u),~~t\geq 0,
\end{align}
 where  $z(t)$ is a nonnegative function, such that 
$a:=\int_{0}^{\infty}z(u)du<\infty$.
Define \begin{align}\label{standardrenewal2}
 G(x)=\int_{0}^{x}\Big(Z(t)-\frac{a}{\mu}\Big)dt.
\end{align}
Then \begin{align}\label{standardrenewal3}
 \lim_{x\rightarrow \infty}G(x)=-\frac{1}{\mu}\int_{0}^{\infty}uz(u)du+a\frac{\sigma^2+\mu^2}{2\mu^2}.
\end{align}
\end{thm}

\begin{proof} Let $V(t)$ be the standard renewal function in (\ref{stand}), where $F_m(t)=\mathbf{P}\Big(\sum_{k=0}^{m}X_k\leq t\Big)$.
By applying \cite[Theorem IV.2.4]{asmussen}, 
\begin{align}\label{standardrenewal4}
 Z(t)=\int_{0}^{t}z(t-u)dV(u)=\int_{0}^{\infty}z(u)dV(t-u),
\end{align}
 where the last equality follows because $V(t)=0$ for $t\leq 0$.
By applying (\ref{standardrenewal4}) and Fubini's Theorem we get \begin{align}\label{standardrenewal5}&G(x)=
 \int_{0}^{\infty}z(u)\int_{0}^{x}dV(t-u)du-\frac{ax}{\mu}=\int_{0}^{\infty}z(u)V(x-u)du-\frac{ax}{\mu}.
\end{align}
Hence,
\begin{align}\label{standardrenewal6}&G(x)=\int_{0}^{\infty}z(u)\Big(V(x-u)-\frac{x}{\mu}\Big)du\nonumber\\&=-\frac{1}{\mu}\int_{0}^{x}z(u)udu-
\frac{1}{\mu}\int_{x}^{\infty}z(u)xdu+\int_{0}^{x}z(u)\Big(V(x-u)-\frac{x-u}{\mu}\Big)du.
\end{align}

From \cite[Proposition VI.4.1]{asmussen} we have $V(t)-\frac{t}{\mu}\rightarrow \frac{\sigma^2+\mu^2}{2\mu^2}$ and by \cite[Proposition VI.4.2]{asmussen}, $0\leq V(t)-\frac{t}{\mu}\leq \frac{\sigma^2+\mu^2}{\mu^2}$. Hence, the Lebesgue dominated convergence theorem applied to the last integral in (\ref{standardrenewal6}) gives
 \begin{align}\label{standardrenewal7}&
 \lim_{x\rightarrow \infty } \int_{0}^{\infty}z(u)\Big(V(x-u)-\frac{x-u}{\mu}\Big)I\{u\leq x\}du=\int_{0}^{\infty}z(u)\frac{\sigma^2+\mu^2}{2\mu^2}du.\end{align}
Note that for all $x$, $\int_{x}^{\infty}z(u)(u-x)du\geq 0$.  Thus, if $\int_{0}^{\infty}z(u)udu$ is integrable 
$\lim_{x\rightarrow \infty }\int_{x}^{\infty}z(u)xdu=0$, and the convergence result in (\ref{standardrenewal3}) obviously follows. If $\int_{0}^{\infty}z(u)udu$ is not integrable then we have a special case of (\ref{standardrenewal3}), i.e., $\lim_{x\rightarrow \infty} G(x)=-\infty$.
\end{proof}

We define the function 
\begin{align}\label{V1}
 W(x)=\int_{0}^{x}e^{-t}(U(t)-\mu^{-1}e^t)dt.
\end{align}

The next result is a corollary of Theorem \ref{thm6}.
\begin{cor}\label{cor3} The function $W(x)$
in (\ref{V1}) satisfies  \begin{align} \label{V}W(x)=\frac{\sigma^2-\mu^2}{2\mu^2}-\mu^{-1}+o(1)~~as ~x\rightarrow \infty. 
\end{align}
\end{cor}

\begin{proof} 

We apply Theorem \ref{thm6} to $Z(t)=\widehat{U}(t)=e^{-t}U(t)$  defined in the proof of Lemma \ref{lem5} (recall that $\widehat{U}(t)$ satisfies the renewal equation in (\ref{ny})).
Now, the constant $a$ as defined in Theorem \ref{thm6}, satisfies $a=\int_{0}^{\infty}\widehat{\nu}(u)du$, thus from (\ref{renewal equation4}) and (\ref{reneq}) we get $a=1$.
Using  (\ref{expren}) and (\ref{renewal equation4})--(\ref{reneq})  gives,
\begin{align*}
 \int_{0}^{\infty}\widehat{\nu}(u)udu=\int_{0}^{\infty}e^{-u}{\nu}(u)udu=\int_{0}^{\infty}e^{-u}{\nu}(u)du+
\int_{0}^{\infty}ue^{-u} d{\nu}(u)=1+\mu.
\end{align*}

\end{proof}

\section{Proofs of the Main Results}

\subsection{Proof of Theorem \ref{assumption1}}\label{1}

\subsubsection{Lemmas of Theorem \ref{assumption1}}\label{lemmas}


We present below some crucial lemmas by which we can then prove Theorem \ref{assumption1}.
 The proofs of these lemmas are given in Section \ref{proofslemmas} below.
The first lemma is fundamental for the proof.
\begin{Lemma}\label{lem11} For the first moment of the number of vertices $N$ we have
\begin{align}\label{obvious}\mathbf{E}(N)=\mathcal{O}(n)
\end{align}
and for the second moment of $N$ we have
\begin{align}\label{obv}\mathbf{E} (N^2)=\mathcal{O}\Big(n^2\Big).
\end{align} 
\end{Lemma} 

 \begin{Lemma}\label{lem14} Adding $K$ balls to a tree will only affect the expected number of nodes in a split tree by $\mathcal{O}(K)$ nodes.
\end{Lemma}

Let $R$ be the set of vertices such that conditioned on the split vectors, $r\in R$, if $M_{r}^{n}:=n\prod_{j=1}^{d(r)}W_j< B$ and $M_{p(r)}^{n}:=n\prod_{j=1}^{d(r)-1}W_j\geq B$, recall that $p(r)$ is the parent of $r$. For now we just let $B$ be large; however, later our choice of $B$ will be more precise.
 To show (\ref{vertexball}) we consider all subtrees rooted at some vertex $r\in R$. We denote these subtrees by $T_{r,B},~r\in R$.
Recall from (\ref{sub}) that with ``large'' probability the cardinality $n_r$ is ``close'' to $M_{r}^{n}$. We will show that in fact we can replace $n_r$ by $M_{r}^{n}$ in our calculations.  Let $n_r$ be the number of balls and let $N_r$ be the number of nodes in the $T_{r,B}$ subtree.
Corollary \ref{lem6,1} implies that most vertices are in the $T_{r,B}$ subtrees, i.e., \begin{align}\label{haha}\mathbf{E}(N)=\mathbf{E}\Big(\sum_{r\in R} N_{r}\Big)+\mathcal{O}\Big(\frac{n}{B}\Big).
 \end{align}

The next lemma shows that the expected number of vertices in the $T_{r,B}$ subtrees with subtree sizes $n_{r}$ that differ significantly from $M_{r}^{n}$ is bounded by a ``small'' error term for large $B$. Since the variance of a Bin$(m,p)$ distribution is $m(p-p^2)$,
the Chebyshev and Markov inequalities give similarly as in (\ref{iliada1,1}) 
that for large $B$,
\begin{align} \label{skola}\mathbf{P}\Big(|n_{r}-M_{r}^{n}|\geq B^{0.6} \Big)
&\leq4\dfrac{\mathbf{E}\Big(M_{r}^{n}\Big)}{B^{1.2}}+\frac{\sum_{k=1}^{\infty}4sb^{-k}}{B^{0.6}}
\leq\frac{1}{B^{0.1}}.\end{align} From (\ref{haha}) we have
\begin{multline}\label{haj}\mathbf{E}(N)=\mathbf{E}\Big(\sum_{r\in R} N_{r}I\{|n_{r}-M_{r}^{n}|\geq B^{0.6}\}\Big)+\\
\mathbf{E}\Big(\sum_{r\in R} N_{r}I\{|n_{r}-M_{r}^{n}|\leq B^{0.6}\}\Big)+\mathcal{O}\Big(\frac{n}{B}\Big).
 \end{multline}

\begin{Lemma}\label{lem16} 
The expected value of the number of nodes that are not in the $T_{r,B},~r\in R$, subtrees with subtree size $n_r$ that differs from $M_{r}^{n}$ with at least $B^{0.6}$ balls, is 
\begin{align}\label{oj}
&\mathbf{E}(\sum_{r\in R} N_{r}I\{|n_{r}-M_{r}^{n}|\geq B^{0.6}\})=\mathcal{O}\Big(\frac{n}{B^{0.1}}\Big),\end{align} hence, from (\ref{haj}) \begin{align}\label{haj1}\mathbf{E}(N)=\mathbf{E}\Big(\sum_{r\in R} N_{r}I\{|n_{r}-M_{r}^{n}|\leq B^{0.6}\}\Big)+\mathcal{O}\Big(\frac{n}{B^{0.1}}\Big).\end{align}
\end{Lemma}

We also sub-divide the $T_{r,B},~r\in R$, subtrees into smaller classes, wherein the $M_{r}^{n}$'s in each class are close to each-other. 
 Choose  $\gamma:=\epsilon^2$ and let $ Z:=\{B,B-\gamma B,B-2\gamma B,\dots,\epsilon B\}$, where  $\epsilon=\frac{1}{k}$ for some positive integer $k$.
We write $R_z\subseteq R,~z\in Z$, 
for the set of vertices $r\in R$, such that  $M_{r}^{n}\in [z-\gamma B,z)$ and $M_{p(r)}^{n}\geq B$. 
(Note that the intervals are of length $\gamma B$ and that the set $Z$ contains at most  $\frac{1}{\gamma}$ elements.) 
We write $|R_z|$ for the number of nodes in $R_z$. The next lemma is a result that we get by the use of renewal theory applied to the renewal function $U(t)$
in (\ref{renewal function}).
\begin{Lemma}\label{lem12} 
Let $S:=\{1, 1-\gamma,1-2\gamma,\dots ,\epsilon\}$, where $\gamma=\epsilon^2$.
Choose  $\alpha\in S$ and let
$\frac{n}{B}\rightarrow\infty$, then
\begin{align}\label{residual4.2}\dfrac{\mathbf{E}(|R_{\alpha B}|)}{\frac{n}{B}}=c_{\alpha} +o(1),
\end{align} for a constant $c_{\alpha}$ (only depending on $\alpha$), and also $\sum_{\alpha \in S}c_{\alpha}=\mathcal{O}\big(1\big)$, where the constant in $\mathcal{O}$ is not depending on $\epsilon$.
\end{Lemma} 

 
Before proving these lemmas we show how their use leads to the proof of Theorem \ref{assumption1}.
\subsubsection{Proof of (\ref{vertexball}) in Theorem \ref{assumption1}}
\begin{proof}

For showing (\ref{vertexball}) it is enough to show that for two arbitrary values of  the cardinality $n$ and $\widehat{n}$, where $\widehat{n}\geq n$, we have 
\begin{align}\label{limsup}\Big|\frac{\mathbf{E}(N)}{n}-\frac{\mathbf{E}(\widehat{N})}{\widehat{n}}\Big|=\mathcal{O}\big(\epsilon\big)~~~~\mathrm{as}~n~\rightarrow\infty,~~\forall\epsilon>0.
\end{align}
Since (\ref{limsup}) implies that $\frac{\mathbf{E}(N)}{n}$ is Cauchy it follows that $\frac{\mathbf{E}(N)}{n}$ converges to some constant $\alpha$ as $n$ tends to infinity; hence, we deduce (\ref{vertexball}). 

We will now prove (\ref{limsup}). 

Recall from Section \ref{lemmas} that we will consider the subtrees 
$T_{r,B},~r\in R$, rooted at $r$; these are defined such that $M_{r}^{n}:=n\prod_{j=1}^{d(r)}W_j< B$ and $M_{p(r)}^{n}:=n\prod_{j=1}^{d(r)-1}W_j\geq B$.

Let $R'\subseteq R$ be the set of vertices such that $r\in R'$ if
\begin{align} \label{epsilon}|n_{r}-M_{r}^{n}|\leq B^{0.6}.
 \end{align}
Lemma \ref{lem16}  shows that we only need to consider the vertices in $r\in R'$.


Let $ R''\subseteq R'$ be the set of vertices such that $r\in R''$ if 
$r\in R'$ and
\begin{align}\label{epsilon2} 
\epsilon B< M_{r}^{n}< B.
\end{align}
We will now explain that it is enough to consider the vertices $r\in R''$.

Corollary \ref{lem6,1} for $K=B$ gives that the expected number of parents $p(r)$ such that $M_{p(r)}^{n}\geq B$ is $\mathcal{O}\Big(\frac{n}{B}\Big)$; thus, since they only have $b$ children each, also the expectation of $|R|$ is $\mathcal{O}\Big(\frac{n}{B}\Big)$. Hence, for $r\in R'$  by using (\ref{obvious}) in Lemma \ref{lem11}, we get that the  expected number of nodes in the $T_{r,B}$, $r\in R'$,  with $ M_{r}^{n}\leq \epsilon B$ is bounded by $\mathcal{O}\big(\epsilon n\big)$. 

From (\ref{haj1})in Lemma \ref{lem16}, we get \begin{align}\label{limsupha}\mathbf{E}(N)&=\mathbf{E}\Big(\sum_{r\in R''}N_r\Big)+\mathcal{O}\big(\epsilon n\big)+\mathcal{O}\Big(\frac{n}{B^{0.1}}\Big).
\end{align} 
Recall that  we sub-divide the $T_{r,B},~r\in R$, subtrees into smaller classes, wherein the $M_{r}^{n}$'s in each class are close to each-other, by introducing the subsets $R_z\subseteq R,~z\in Z$, where $ Z=\{B,B-\gamma B,B-2\gamma B,\dots,\epsilon B\}$.
Hence, (\ref{limsupha}) gives 
\begin{align}\label{limsupnu}&\mathbf{E}(N)
=\mathbf{E}\Big(\sum_{z\in Z}\sum_{r\in R'\cap  R_{z}}N_r\Big)+
\mathcal{O}\Big(\epsilon n\Big)+\mathcal{O}\Big(\frac{n}{B^{0.1}}\Big).
\end{align}

We will now apply Lemma \ref{lem14} to calculate the expected value in (\ref{limsupnu}). Let $r_z$ be an arbitrarily chosen node in $R'\cap R_z$, where $z\in Z$.
By using (\ref{epsilon}) and Lemma \ref{lem14}, for any node $r_z \in R'\cap R_z$, we get that the expected number of nodes in a tree with the number of balls in an interval $[z-\gamma B,z)$  is equal to
$\mathbf{E}\Big(N_{r_z}\Big)+\mathcal{O}\Big(\gamma B\Big)$.
By using (\ref{limsupnu}) this implies that  
\begin{align}\label{limsupnu.1}&\mathbf{E}(N)=\sum_{z\in Z}\mathbf{E}(|R'\cap R_z|)\Big(\mathbf{E}(N_{r_z})+\mathcal{O}\big(\gamma B\big)\Big)+\mathcal{O}\big(\epsilon n\big)+\mathcal{O}\Big(\frac{n}{B^{0.1}}\Big).
\end{align}
Define $a_x$ as the quotient of the expected number of vertices  in a tree with cardinality $\lfloor x \rfloor$ divided by $\lfloor x \rfloor$. Note from Lemma \ref{lem11} that $a_x=\mathcal{O}\big(1\big)$.

Recall from Lemma \ref{lem12} that $S=\{1, 1-\gamma,1-2\gamma,\dots ,\epsilon\}$. By using  (\ref{residual4.2}) in Lemma \ref{lem12} and applying (\ref{skola}) we have that for each choice of
$\gamma$ and $\alpha \in S$, there is a $\sigma_{\gamma}$ such that for a constant $c_{\alpha}$ (depending on $\alpha$),
\begin{align}\label{residual4.1}\Big|\dfrac{\mathbf{E}(|R'\cap R_{\alpha B}|)}{\frac{n}{B}}-c_{\alpha} \Big|\leq \gamma^{2}+\mathcal{O}\Big(\frac{1}{B^{0.1}}\Big),
\end{align}
whenever
 $\frac{n}{B}\geq \frac{1}{\sigma_{\gamma}}$. We now choose $B=\ln n$, where $n$ is the smallest of the two arbitrary values we start with (i.e., $\widehat{n}\geq n$).
Thus, we have by the choice of $B$ (for $n$ large enough) that $\frac{n}{B}\geq \frac{1}{\sigma_{\gamma}}$ so that (\ref{residual4.1}) holds.

Note that since $\sum_{\alpha \in S}c_{\alpha}=\mathcal{O}(1)$, we have that $\sum_{\alpha \in S}c_{\alpha}\frac{\mathcal{O}(B\gamma)}{B}=\mathcal{O}(\gamma)$.
Recall that  $\gamma:=\epsilon^2$.
Thus, for a constant $c_{\alpha}$ (depending on $\alpha$)  and $a_{\alpha B}=\mathcal{O}\big(1\big)$, we get from (\ref{limsupnu.1}) and (\ref{residual4.1}) that
\begin{align}\label{limsup.1}\mathbf{E}(N)&=n\sum_{\alpha \in S}c_{\alpha}\frac{1}{B}\Big(\alpha B a_{\alpha B} +\mathcal{O}(B\gamma)\Big)+n\sum_{\alpha \in S}\mathcal{O}\big(a_{\alpha B}\gamma^{2}\big)+\mathcal{O}\big(\epsilon n\big)=
\nonumber\\&=n\sum_{\alpha \in S}\alpha a_{\alpha B}c_{\alpha}+\mathcal{O}\big(\epsilon n\big).
\end{align}
In analogy we also get for $\widehat{n}\geq n$,
\begin{align}\label{limsup.1.2}&\mathbf{E}(\widehat{N})=
\widehat{n}\sum_{\alpha \in S}\alpha a_{\alpha B}c_{\alpha}+\mathcal{O}\big(\epsilon\widehat{n}\big).
\end{align}
Thus, (\ref{limsup}) follows, which shows (\ref{vertexball}).

\end{proof}

\subsubsection{Proof of (\ref{vertexball2}) in Theorem \ref{assumption1}}
\begin{proof}

First note that (\ref{obv}) in Lemma \ref{lem11} implies that $\mathbf{Var}(N)=\mathcal{O}\Big(n^2\Big)$. 
The purpose is to use the variance formula 
\begin{align}\label{varianceformula}\mathbf{Var}(Y)=\mathbf{E}\Big(\mathbf{Var}(Y|\mathscr{G})\Big)+
\mathbf{Var}\Big(\mathbf{E}(Y|\mathscr{G})\Big),
\end{align}
where $Y$ is a random variable and $\mathscr{G}$ is a sub-$\sigma$-field, see e.g.\cite[exercise 10.17-2]{gut}.
We consider  the subtrees $T_i,~1\leq i \leq b^D$ at depth $D=c\ln n$, choosing the constant  $c$ small enough so that the number of nodes $Z_D$ between depth $D$ and the root is $\mathcal{O}\big(n^{\epsilon}\big)$ for some arbitrary small $\epsilon$. Let $n_i$ be the number of balls and $N_i$ the number of nodes in $T_i$. 
Conditioned on $\Omega_D$, $N_i,~1\leq i\leq b^D$, are independent and it follows that,
\begin{align}\label{condvar}&\mathbf{Var}(N{|}\Omega_D)= \mathbf{Var}\Big(\sum_{i=1}^{b^D}N_i+Z_D|\Omega_D\Big)=\sum_{i=1}^{b^D}\mathbf{Var}(N_i|\Omega_D)=\sum_{i=1}^{b^D}\mathcal{O}(n_i^{2}).
 \end{align}
Taking expectation in (\ref{condvar}) gives
\begin{align}\label{condvar2}&\mathbf{E}\Big(\mathbf{Var}(N{|}\Omega_D)\Big)
=\sum_{i=1}^{b^D}\mathcal{O}(\mathbf{E}(n_i^{2})).
 \end{align}
Recall that $\Omega_D$ is the $\sigma$-field generated by $\{n_v,~d(v)\leq D\}$. 
\begin{Lemma}\label{lem10} For $D=c\ln n$ there is a $\delta>0$, such that
\begin{align}\label{subtreei}\mathbf{E}\Big(\sum_{i=1}^{b^D}n_i^2\Big)=\mathcal{O}(n^{2-\delta}).
 \end{align}
\end{Lemma} 

\begin{proof}

The representation of subtree sizes in split trees described in (\ref{binomial}) in Section \ref{Subtrees} gives 
in particular that conditioning on $\mathscr{G}_D$, $n_i$ for $i$ at depth $D$  is bounded from above (in stochastic sense) by 
\begin{align}\label{binomial5}
 \mathrm{Bin}(n,\prod_{j=1}^{D}W_j)+s_1D,
\end{align}
where $W_{j},j\in \{1,\dots, D\}$, are i.i.d.\ random variables distributed as $V$.
The fact that the second moment of a $\mathrm{Bin}(m,p)$ is  $m^2p^2+mp-mp^2$ and the bound of $n_i$ in (\ref{binomial5}) give 
\begin{align*}
 \mathbf{E}({n_{i}}^2|\mathscr{G}_{D})\leq n^{2}\prod_{j=1}^{D}W_j^{2}+\mathcal{O}(n D\prod_{j=1}^{D}W_j)+\mathcal{O}(D^2).
\end{align*}
Note that $\mathbf{E}(W_j^2)<\mathbf{E}(W_j)=\frac{1}{b}$, since $W_j\in (0,1)$.  Hence, there is an $\epsilon>0$ such that
\begin{align}\label{subtree2,2}\mathbf{E}({n_{i}}^2)&\leq n^2\prod_{j=1}^{D} {\mathbf{E}({W_j}^2})+\mathcal{O}\Big(\frac{nD}{b^D}\Big)+\mathcal{O}(D^2)\nonumber
\\&\leq\frac{n^2}{(b+\epsilon)^{D}}+\mathcal{O}\Big(\frac{nD}{b^D}\Big)+\mathcal{O}(D^2),\end{align}
and thus there is a $\delta>0$ such that
\begin{align*}\mathbf{E}\Big(\sum_{i=1}^{b^D}n_i^2\Big)=\mathcal{O}(n^{2-\delta}),
 \end{align*}
which shows (\ref{subtreei}).
\end{proof}

Thus, (\ref{condvar2}) and (\ref{subtreei}) in Lemma \ref{lem10} give
\begin{align}\label{condvar3}&\mathbf{E}\Big(\mathbf{Var}(N{|}\Omega_D)\Big)=\mathcal{O}(n^{2-\delta}).
 \end{align}

By applying (\ref{vertexball}) in Theorem \ref{assumption1} gives
\begin{align}\label{vantevarde}&\mathbf{E}(N{|}\Omega_D)=\sum_{i=1}^{b^D}\Big(\alpha n_i+o(n_i)\Big) +\mathbf{E}(Z_D{|}\Omega_D).
 \end{align}
Applying (\ref{vantevarde}) gives
\begin{align}\label{condvar4}&\mathbf{Var}\Big(\mathbf{E}(N{|}\Omega_D)\Big)
=\mathbf{Var}\Big(\alpha n+o(n)\Big)=o(n^2).
 \end{align}

 Thus, by applying the variance formula in (\ref{varianceformula}) 
we get from (\ref{condvar3}) and (\ref{condvar4}) that $\mathbf{Var}(N)=o(n^2)$.
\end{proof}

\begin{rem} 
The proof shows that if
we can improve the result in (\ref{vertexball}) in Theorem \ref{assumption1} such that $\mathbf{E}(N)=\alpha n+\mathcal{O}(n^{1-c_1})$ for some constant $c_1>0$, we will also get a sharper result for the variance, i.e., $\mathbf{Var}(N)=o(n^{2-c_2})$ for some constant $c_2>0$.
\end{rem}

\subsubsection{Proofs of the Lemmas of Theorem \ref{assumption1}}\label{proofslemmas}

\begin{proof}[Proof of Lemma \ref{lem11}]
 (Note that if $s_0>0$ it is always true that $N\leq n$ and if $s_1>0$ we always have $N\leq 2n$.)
For $s_0=s_1=0$ we can argue as follows: 
When a new ball is added to the tree the expected number of additional nodes is bounded by the expected number of nodes one gets from a splitting node.
Let $Z$ be the number of nodes that one gets when a node of $s+1$ balls splits.
Then
\begin{align}\label{bound}
\mathbf{E}(Z)=\sum_{k=1}^{\infty}k\mathbf{P}(Z=k).
\end{align}
Note that once a node gives balls to at least 2 children the splitting process ends.
Thus, \begin{align*}
       \mathbf{P}\Big(Z=k\big{|}\mathscr{G}_k\Big)=\mathcal{O}\Big(\sum_{v,~d(v)=k} \prod_{j=1}^{k}W_{j,v}^{s+1}\Big).
      \end{align*}
Hence, (\ref{bound}) implies,
\begin{align}\label{bound2}&\mathbf{E}(Z)=
\sum_{k=1}^{\infty}k\mathcal{O}\Big((b\mathbf{E}(V^{s+1}))^k\Big).
\end{align}
There is a $\delta>0$ such that $b\mathbf{E}(V^{s+1})\leq b^{-\delta}$ since 
\begin{align}\label{axel2}
\mathbf{E}(V^{s+1})<\mathbf{E}(V)=\frac{1}{b},
\end{align} 
for $V\in (0,1)$. Thus, (\ref{bound2}) gives
\begin{align}\label{bound2,1}\mathbf{E}(Z)&=
\sum_{k=1}^{\infty}k\mathcal{O}\Big(b^{-k\delta}\Big)=
\mathcal{O}\Big(\frac{b^{-\delta}}{(1-b^{-\delta})^2}\Big)=\mathcal{O}(1).
\end{align}
This shows (\ref{obvious}). 

Now we show (\ref{obv}).
Note that (\ref{obv}) obviously holds if $s_1>0$ or $s_0>0$, since then $N\leq 2n$.
Recall that $Z$ is the number of nodes that one gets when a node of $s+1$ balls splits.
Then by the well-known Minkowski's inequality 
\begin{align}\label{bound3}
\mathbf{E}(N^2)\leq n^2 \mathbf{E}(Z^2).
\end{align}
By similar calculations as in (\ref{bound2})--(\ref{bound2,1}) we get that for some constant $\delta>0$, 
\begin{align}\label{bound4}
\mathbf{E}(Z^2)&\leq \sum_{k=1}^{\infty}k^2\mathbf{P}(Z=k)=
\sum_{k=1}^{\infty}k^2\mathcal{O}\Big(b^{-k\delta}\Big)=\mathcal{O}(1)
.\end{align}
Thus, (\ref{obv}) follows from (\ref{bound3}) and (\ref{bound4}).

\end{proof}

\begin{proof}[Proof of Lemma \ref{lem11}]
The proof of this lemma is in analogy with the proof of (\ref{obvious}) in 
Lemma \ref{lem11}. Adding one ball to the tree will only increase the vertices if it is added to a leaf with $s$ balls. Recall that $Z$ is the number of nodes that one gets when a node of $s+1$ balls splits. Hence, (\ref{bound2}) gives $\mathbf{E}(Z)=\mathcal{O}\big(1\big)$, implying that $K$ balls can only create $\mathcal{O}\big(K\big)$ additional nodes.
\end{proof}

\begin{proof}[Proof of Lemma \ref{lem16}]
By applying $(\ref{skola})$ we get that with probability at least $1-\frac{1}{B^{0.1}}$,
\begin{align} \label{epsilonhej}|n_{r}-M_{r}^{n}|\leq B^{0.6}.
 \end{align}

We have 
\begin{align}\label{hoa}&\mathbf{E}\Big(\sum_r n_{r}I\{|n_{r}-M_{r}^{n}|\geq B^{0.6}\}\Big)=E_1+E_2,
 \end{align}
where
\begin{align*}E_1&=\mathbf{E}\Big(\sum_r n_{r}I\{|n_{r}-M_{r}^{n}|\geq B^{0.6}\}I\{n_{r}\leq 2M_{r}^{n}\}\Big),\nonumber\\E_2&=\mathbf{E}\Big(\sum_r n_{r}I\{|n_{r}-M_{r}^{n}|\geq B^{0.6}\}
I\{n_{r}> 2M_{r}^{n}\}\Big).
 \end{align*}  Hence, the facts that $\sum_r M_{r}^{n}=\mathcal{O}\big(n\big)$ and that the bound in (\ref{epsilonhej}) holds with probability $1-\frac{1}{B^{0.1}}$,  give
\begin{align*}E_1&\leq \mathbf{E}\Big(\sum_r 2M_{r}^{n}I\{|n_{r}-M_{r}^{n}|\geq B^{0.6}\}\Big)=\mathcal{O}\Big(\frac{n}{B^{0.1}}\Big).
 \end{align*}

Recall that $R$ is the set of vertices such that $r\in R$, if $r$ is the root of a $T_{r,B}$ subtree.
We obviously have 
\begin{align*}E_2&\leq \mathbf{E}\Big(\sum_v 2(n_{v}-M_{v}^{n})I\{n_{v}> 2M_{v}^{n}\}I\{v\in R\}\Big).
 \end{align*}

By summing over vertices $v$ at depth $k$ we get 
\begin{align}\label{hej1}
E_2\leq \sum_{k=0}^{\infty}2b^k\mathbf{E}\Big( (n_{v}-M_{v}^{n})I\{n_{v}>2M_{v}^{n}\}\Big)
\mathbf{P}\Big(v\in R\Big).
\end{align}
We write $F$ for the expected value in (\ref{hej1}), i.e., \begin{align*}F:=\mathbf{E}\Big((n_{v}-M_{v}^{n})I\{n_{v}> 2M_{v}^{n}\}\Big).
 \end{align*}
Hence, the conditional Cauchy-Schwarz and the conditional Markov inequalities give
\begin{align}\label{ej1}  F:&\leq \mathbf{E}\left(\sqrt{\mathbf{E}\Big( \Big(n_v-M_{v}^{n}\Big)^2\Big|\mathscr{G}_d\Big)}\sqrt{\mathbf{P}\Big(n_{v}> 2M_{v}^{n}\Big|\mathscr{G}_d\Big)}\right)\nonumber\\&\leq \min\left(\mathbf{E}\left(\frac{\mathbf{E}\Big( \Big(n_v-M_{v}^{n}\Big)^2\Big|\mathscr{G}_d\Big)}{M_{v}^{n}}\right),\mathbf{E}\left(\sqrt{\mathbf{E}\Big( \Big(n_v-M_{v}^{n}\Big)^2\Big|\mathscr{G}_d\Big)}\right)\right).
 \end{align}
From (\ref{binomial}) we have that  for all $v$ with $d(v)=d$, conditioned on $\mathscr{G}_d$ (i.e., the $\sigma$-field generated by $W_{j,v},j\in\{1,\dots, d\}$), $n_v\leq n'_{v}+n''_{v}$, where
\begin{align*}n'_{v}&:=\mathrm{Bin}(n,\prod_{j=1}^{d}W_{j,v}),\\ n''_{v}&:=\mathrm{Bin}(s_1,\prod_{j=2}^{d}W_{j,v})+\dots + \mathrm{Bin}(s_1,W_{d,v})+s_{1}.
 \end{align*}
Thus, (\ref{ej1}) gives for $M_{v}^{n}\geq 1$,
\begin{align}\label{tal}F&\leq \mathbf{E}\left(\frac{\mathbf{E}\Big( \Big(n_v-M_{v}^{n}\Big)^2\Big|\mathscr{G}_d\Big)}{M_{v}^{n}}\right)\nonumber
\\&\leq \mathbf{E}\left(\frac{\mathbf{E}\Big( \Big(n'_v-M_{v}^{n}\Big)^2\Big|\mathscr{G}_d\Big)}{M_{v}^{n}}+
 \frac{\mathbf{E}\Big( (n''_v)^{2}+2n'_vn''_v-2n''_vM_{v}^{n}\Big|\mathscr{G}_d\Big)}{M_{v}^{n}}\right)\nonumber
\\&\leq \mathbf{E}\left(\frac{\mathbf{E}\Big( \Big(n'_v-M_{v}^{n}\Big)^2\Big|\mathscr{G}_d\Big)}{M_{v}^{n}}\right)+
\mathbf{E}\Big( (n''_v)^{2}\Big)+2\mathbf{E}\Big(n''_v\Big),
\end{align} 
where we in the last equality apply that $\mathbf{E}\Big(n'_v\Big)=M_{v}^{n}$.
For $M_{v}^{n}< 1$ we apply that (\ref{ej1}) gives 
\begin{align}\label{tal2}F\leq \mathbf{E}\left(\sqrt{\mathbf{E}\Big( \Big(n_v-M_{v}^{n}\Big)^2\Big|\mathscr{G}_d\Big)}\right).
 \end{align}
By applying the fact that the variance of a Bin$(m,p)$ distribution is $m(p-p^2)$ we get
$\mathbf{E}\Big( \Big(n'_v-M_{v}^{n}\Big)^2\Big|\mathscr{G}_d\Big)\leq M_{v}^{n}$, and 
from the Minkowski's inequality we easily deduce that $\mathbf{E}\Big( (n''_v)^{2}\Big)=\mathcal{O}(1)$.
Hence, by using that we can bound $F$ as in (\ref{tal})  for $M_{v}^{n}\geq 1$, and by the bound in (\ref{tal2}) for $M_{v}^{n}< 1$, we get that $F=\mathcal{O}(1)$.
Thus, from (\ref{hej1}) we get \begin{align}\label{ej2}
E_2\leq \sum_{k=0}^{\infty}b^k\mathcal{O}\big(1\big)\mathbf{P}(v\in R).
\end{align}
Note that $v\in R$ only if $M_{p(v)}^{n}\geq B$. Hence, by applying Corollary  \ref{lem6,1} for $K=B$, we get from (\ref{ej2}) that $E_2=\mathcal{O}\big(\frac{n}{B}\big)$.
By applying Lemma \ref{lem11} in combination with (\ref{hoa}) and using the bounds of $E_1$ and $E_2$  we get 
\begin{align}\label{hoho}\mathbf{E}\big(\sum_r N_rI\{|n_{r}-M_{r}^{n}|\geq B^{0.6}\}\big)&=\mathcal{O}\Big(\mathbf{E}\big(\sum_r n_rI\{|n_{r}-M_{r}^{n}|\geq B^{0.6}\}\big)\Big)
\nonumber\\&=\mathcal{O}\Big(\frac{n}{B^{0.1}}\Big).
 \end{align}

\end{proof}

\begin{proof}[Proof of Lemma \ref{lem12}]

Recall the definition of $Y_k=-\sum_{j=1}^{k}\ln W_j$ and $\nu(t)=b\mathbf{P}(-\ln W_j\leq t)$. Also recall that we write $S=\{1, 1-\gamma,1-2\gamma,\dots ,\epsilon\}$ for $\gamma=\epsilon^2$.
We have for $\alpha \in S$
\begin{align*}&\mathbf{E}(|R_{\alpha B}|)=\sum_{k=0}^{\infty}b^{k+1} \Big(\mathbf{P}\Big(\{Y_k-\ln W_{k+1}> \ln \frac{n}{B}-\ln \alpha\}\bigcap\{Y_k\leq\ln \frac{n}{B}\}\Big)\nonumber\\&-\mathbf{P}\Big(\{Y_k- \ln W_{k+1}>\ln \frac{n}{B}-\ln\Big( \alpha-\gamma\Big)\}\bigcap\{Y_k\leq \ln \frac{n}{B}\}\Big)\Big).
\end{align*}
We write $q:=\ln \frac{n}{B}$.
From the definition of $U(t)$ we have that $\mathbf{E}(|R_{\alpha B}|)$  is equal to
\begin{multline*}Z(q):=\int_{0}^{q}b
\Big(\mathbf{P}\Big(- \ln W_{k+1}>q-t-\ln \alpha\Big)\\-\mathbf{P}\Big(- \ln W_{k+1}> q-t-\ln\Big( \alpha-\gamma\Big)\Big)\Big)dU(t).
\end{multline*}
Hence,
\begin{multline}\label{hopp}Z(q):=\int_{0}^{q}b\mathbf{P}\Big(q-t-\ln \alpha<- \ln W_{k+1}\leq q-t-\ln\Big( \alpha-\gamma\Big)\Big)dU(t).
\end{multline}
We write
 \begin{align*}G(t):=b
\mathbf{P}\Big(t-\ln \alpha<- \ln W_{k+1}\leq t-\ln\Big( \alpha-\gamma\Big)\Big).
\end{align*}
Thus,
 \begin{align*}Z(q)=(G*dU)(q).
\end{align*}

Recall that we write $d\omega(t)=e^{-t}d\nu(t)$ where $\omega(t)$ is a probability measure.
Recall from (\ref{ny}) that we have
 \begin{align*} \widehat{U}(t)=\widehat{\nu}(t)+(\widehat{U}*d\omega)(t),
\end{align*}
where $\widehat{U}(t):=e^{-t}U(t)$ and $\widehat{\nu}(t):=e^{-t}\nu(t)$.
Thus, by using \cite[Theorem VI.5.1]{asmussen}
we have for $\widehat{Z}(x)=e^{-x}Z(x)$ and  $\widehat{G}(x)=e^{-x}G(x)$ that
 \begin{align*}\widehat{Z}(q)=(\widehat{G}*d\omega)(q).
\end{align*}
By using (\ref{hopp}) this implies that 
\begin{multline}\label{hopp1}\widehat{Z}(q)=\int_{0}^{q}be^{t-q}
\mathbf{P}\Big(q-t-\ln \alpha<- \ln W_{k+1}\leq q-t-\ln\Big( \alpha-\gamma\Big)\Big)d\omega(t).
\end{multline}

By using the key renewal theorem \cite[Theorem II.4.3]{gut2} applied to $\widehat{U}(t)$ we get
\begin{align}\label{hopp2}\lim_{{q\rightarrow\infty}}\widehat{Z}(q)=\frac{b}{\mu}\int_{0}^{\infty}e^{-t}
\mathbf{P}\Big(t-\ln \alpha<- \ln W_{k+1}\leq t-\ln\Big( \alpha-\gamma\Big)\Big)dt.
\end{align}
Note that $\lim_{{q\rightarrow\infty}}\widehat{Z}(q):=c_{\alpha}$, for some constant $c_{\alpha}$ only depending on $\alpha$. 

Thus, by using $\widehat{Z}(x)=e^{-x}Z(x)$ we get that 
\begin{align*}&\mathbf{E}(|R_{\alpha B}|)
=\frac{n}{B}c_{\alpha}+o\big(\frac{n}{B}\big),
\end{align*} for the constant $c_{\alpha}$ (only depending on $\alpha$), which shows (\ref{residual4.2}).

Also note that  we have \begin{align}\label{hopp3}\sum_{\alpha \in S}c_{\alpha}&=\frac{b}{\mu}\int_{0}^{\infty}e^{-t}\sum_{\alpha \in S}
\mathbf{P}\Big(t-\ln \alpha<- \ln W_{k+1}\leq t-\ln\Big( \alpha-\gamma\Big)\Big)dt\\&=
\frac{b}{\mu}\int_{0}^{\infty}e^{-t}
\mathbf{P}\Big(t<- \ln W_{k+1}\leq t-\ln\big( \epsilon-\gamma\big)\Big)dt\leq \frac{b}{\mu}.
\end{align}

\end{proof}

\subsection{Proof of Theorem \ref{lemma1}}\label{2}

\begin{proof}
We use large deviations to show this theorem (in fact we get a sharper bound of the number of bad nodes).
Note that a vertex $v$  belongs to the tree if and only if $n_v\geq 1$. 
Recall that there is an upper bound of $n_v$ with $d(v)=d$ in (\ref{binomial}) above, i.e.,
conditioning on $\mathscr{G}_d$ in stochastic sense, 
\begin{align}\label{bolldel} n_v &\leq \mathrm{Bin}(n,\prod_{j=1}^{d}W_{j,v})+\mathrm{Bin}(s_1,\prod_{j=2}^{d}W_{j,v})\nonumber\\&+\mathrm{Bin}(s_1,\prod_{j=3}^{d}W_{j,v})+\dots + \mathrm{Bin}(s_1,W_{d,v})+s_1 ,
\end{align}
where $W_{j,v},~j\in\{1,\dots, d\}$, are i.i.d.\ random variables distributed as $V$.
It is enough to just consider the first term $\mathrm{Bin}(n,\prod_{j=1}^{d}W_{j,v})$ in (\ref{bolldel}), and prove that the number of bad nodes with $\mathrm{Bin}(n,\prod_{j=1}^{d}W_{j,v})\geq 1$ is bounded by 
$\mathcal O_{L^{1}} \Big(\frac {n}{\ln^{k+1}{n}}\Big)$, where we choose $k$ large enough. 
If $s_1=0$, $\mathrm{Bin}(n,\prod_{j=1}^{d}W_{j,v})$ is the only term in (\ref{bolldel}).
We now explain the fact that we can ignore the terms in $n_v$ that occurs because of the parameter $s_1$.
Assume that for split trees with $s_1=0$, the number of bad nodes is bounded by 
$\mathcal O_{L^{1}} \Big(\frac {n}{\ln^{k+1}{n}}\Big)$. 
We first consider the vertices  with $d\leq \mu^{-1}\ln{n}-\ln^{0.5+\epsilon}{n}$.
If $s_1>0$, we assume that we first add the $n$ balls as in the construction of a split tree with the parameter $s_1=0$.
 Hence,  the number of vertices $v$ with $d\leq \mu^{-1}\ln{n}-\ln^{0.5+\epsilon}{n}$, is bounded by 
$\mathcal O_{L^{1}} \Big(\frac {n}{\ln^{k+1}{n}}\Big)$.
We now repay the subtree sizes for their potential loss of balls because of $s_1>0$.
A vertex $v$ at depth $d$ can at most have a loss of $s_1d$ balls in the subtree rooted at $v$. These balls cannot give more than 
$s_1bd$ nodes to the tree (since only if  $s_0=s_1=0$ it is  possible for an increment of more than $b$ nodes when a new ball is added to the tree). Thus, since $d\leq \mu^{-1}\ln{n}$ and the fact that we assume that we have $\mathcal O_{L^{1}} \Big(\frac {n}{\ln^{k+1}{n}}\Big)$ nodes before the repayment of the loss of balls, these additional balls cannot give more than $\mathcal O_{L^{1}} \Big(\frac {n}{\ln^{k}{n}}\Big)$ nodes. 
Now we consider the vertices with  $d\geq \mu^{-1}\ln{n}+\ln^{0.5+\epsilon}{n}$.
 Again we first distribute the $n$ balls assuming that $s_1=0$, and then repay for the potential loss of balls in the subtrees if $s_1>0$.
First note that for $d=\mathcal{O}(\ln n)$ we can argue as in the previous case. This means that the number of nodes with 
$\mu^{-1}\ln{n}+\ln^{0.5+\epsilon}{n}\leq d\leq K\ln{n}$ for some arbitrary constant $K$ is bounded by $\mathcal O_{L^{1}} \Big(\frac {n}{\ln^{k}{n}}\Big)$.
For larger $d$ we argue as follows:
For any constant $K_1>0$,
\begin{align*} & \mathrm {mBin}(s_1,\prod_{j=2}^{d}W_{j,v})+\mathrm {mBin}(s_1,\prod_{j=3}^{d}W_{j,v})+\dots+ \mathrm {mBin}(s_1,W_{d,v})+s_1 \nonumber\\& \leq \mathrm {mBin}(s_1,\prod_{j=2}^{d}W_{j,v})+\dots +\mathrm {mBin}(s_1,\prod_{j=d-\lfloor K_1\ln n \rfloor}^{d}W_{j,v})+ K_1s_1\ln n.
\end{align*}

The Markov inequality gives,
\begin{align} &\mathbf{P}(\mathrm {mBin}(s_1,\prod_{j=2}^{d}W_{j,v})+\dots +\mathrm {mBin}(s_1,\prod_{j=d-\lfloor K_1\ln n\rfloor}^{d} W_{j,v})\geq 1\Big)\nonumber\\& \leq \mathbf{E}\Big(\mathrm {Bin}(s_1,\prod_{j=2}^{d}W_{j,v})+\dots+\mathrm {Bin}(s_1,\prod_{j=d-\lfloor K_1\ln n\rfloor}^{d}W_{j,v})\Big)=\mathcal O \Big(b^{-K_1\ln n}\Big),\end{align}
where the last equality is obtained by first condition on $\mathscr{G}_d$  and then take the expected value twice.
Thus, the expected number of vertices that gets a repayment of at least $K_1s_1\ln n+2$ balls is bounded by 
$\mathcal O \Big(\frac {n}{b^{K_1\ln n}}\Big)$. Since $s_1>0$, we can assume that $d\leq n$.
Hence, the expected number of balls of this contribution is $\mathcal O \Big(\frac {n^2}{b^{K_1\ln n}}\Big)$; choosing $K_1$ large enough  this number is just $o(1)$ and can thus be ignored.

It remains to prove that if $s_1=0$ the number of vertices $v$, where  $d(v)\leq \mu^{-1}\ln{n}-\ln^{0.5+\epsilon}{n}$ or $d(v)\geq \mu^{-1}\ln{n}+\ln^{0.5+\epsilon}{n}$,  with $n_v\geq 1$ is bounded by $\mathcal O_{L^{1}} \Big(\frac {n}{\ln^{k+1}{n}}\Big)$ for any constant $k$.
Note that an upper bound of the expected number of vertices at depth $d$  is given by \begin{align}\label{upperboundprofile} b^d \mathbf{P}(n_v\geq 2),
 \end{align} where $v$ is a vertex at depth $d-1$. Note that this is true even in the case $s_0=0$, since for all internal nodes $n_v\geq s+1$.
Choosing $t>0$, an application of the Markov inequality implies that
\begin{align}\label{cannes}\mathbf{P}(n_v\geq 2)\leq \mathbf{P}(n_v(n_v-1)\geq 2)&\leq\nonumber\\\mathbf{P}(n_v^{t}(n_v-1)^t\geq 2^t)&\leq \frac{\mathbf{E}(n_v^{t}(n_v-1)^t)}{2^t}.
\end{align}
Thus, an upper bound of the expected profile for the vertices at depth $d$ is
\begin{align}\label{cannes2}b^d\mathbf{E}(n_v^{t}(n_v-1)^t),\end{align}
where $v$ is a  is a vertex at depth $d-1$.

First we show that the number of vertices $v$ (assuming $s_1=0$) where $d(v)\geq \mu^{-1}\ln{n}+\ln^{0.5+\epsilon}{n}$ is bounded by $\mathcal O_{L^{1}} \Big(\frac {n}{\ln^{k+1}{n}}\Big)$. We prove this by choosing $t=\frac{1+\epsilon(n)}{2}$, where $\epsilon(n)>0$ is a decreasing function of $n$ that we specify below, and show that
 \begin{align}\label{cannes9}\sum_{d=\lfloor \mu^{-1}\ln{n}+\ln^{0.5+\epsilon}{n}\rfloor-1}^{\infty}b^d\mathbf{E}\big(n_v^{\frac{1+\epsilon(n)}{2}}(n_v-1)^{\frac{1+\epsilon(n)}{2}}\big)=\mathcal O\Big(\frac {n}{\ln^{k+1}{n}}\Big).
\end{align}
Let $X_d$ be a mixed binomial $(n,\prod_{j=1}^{d}W_j)$, where $W_j,~j\in{1,\dots,d}$ are i.i.d.\ random variables distributed as $V$.
To show (\ref{cannes9}) it is enough to show that the expected value of 
\begin{align}\label{cannes8}\sum_{d=\lfloor \mu^{-1}\ln{n}+\ln^{0.5+\epsilon}{n}\rfloor-1}^{\infty}b^d\mathbf{E}\Big(X_d^{\frac{1+\epsilon(n)}{2}}(X_d-1)^{\frac{1+\epsilon(n)}{2}} \big{|} \mathscr{G}_d\Big), 
\end{align} is $\mathcal O \Big(\frac {n}{\ln^{k+1}{n}}\Big)$.
That this is enough follows because of the bound of $n_v$ in (\ref{bolldel}), since we assume that $s_1=0$.
Suppose that $\epsilon(n)<1$, thus the Lyapounov inequality (which is a special case of the well-known H\"{o}lder inequality) gives 
\begin{align}\label{binomial4}\mathbf{E}\Big(X_d^{\frac{1+\epsilon(n)}{2}}(X_d-1)^{\frac{1+\epsilon(n)}{2}}\big{|} \mathscr{G}_d\Big)
&\leq (n^2-n)^{\frac{1+\epsilon(n)}{2}}\prod_{j=1}^{d}W_j^{1+\epsilon(n)}\nonumber\\& \leq \Big(n\prod_{j=1}^{d}W_j\Big)^{1+\epsilon(n)}.
\end{align}
Hence, to show (\ref{cannes9}) we deduce from the right hand-side of the second inequality in (\ref{binomial4}) that it is enough to show that 
\begin{align}\label{cannes3}S_1:=\sum_{d=\lfloor \mu^{-1}\ln{n}+\ln^{0.5+\epsilon}{n}\rfloor-1}^{\infty}b^d\big(\mathbf{E}(W_j^{1+\epsilon(n)})\big)^dn^{1+\epsilon(n)}
=\mathcal O\Big(\frac {n}{\ln^{k+1}{n}}\Big).\end{align}
Taylor expansion gives \begin{align}\label{cannes4}&W_j^{1+\epsilon(n)}=W_je^{\epsilon(n)\ln W_j}=W_j\Big(1+\epsilon(n)\ln W_j+ \frac{\ln^2 W_j}{2}\epsilon^2(n)\Big)\nonumber\\&+\mathcal{O}\Big(W_j\epsilon^3(n) \ln ^3W_j\Big).\end{align}
Thus, by taking expectations in (\ref{cannes4}) we get
\begin{align}\label{cannes5}S_1&=\sum_{d=\lfloor \mu^{-1}\ln{n}+\ln^{0.5+\epsilon}{n}\rfloor-1}^{\infty}\bigg(1-\mu\epsilon(n)+\frac{\sigma^2+\mu^2}{2}\epsilon^2(n)+\mathcal{O}\big(\epsilon^3(n)\big)\bigg)^dn^{1+\epsilon(n)}\nonumber\\&=
\sum_{d=\lfloor \mu^{-1}\ln{n}+\ln^{0.5+\epsilon}{n}\rfloor-1}^{\infty}e^{\ln\Big(1-\mu\epsilon(n)+\frac{\sigma^2+\mu^2}{2}\epsilon^2(n)+
\mathcal{O}\big(\epsilon^3(n)\big)\Big)d+\ln{n}(1+\epsilon(n))}\nonumber\\&=
\sum_{d=\lfloor \mu^{-1}\ln{n}+\ln^{0.5+\epsilon}{n}\rfloor-1}^{\infty}
e^{\Big(-\mu\epsilon(n)+\frac{\sigma^2}{2}\epsilon^2(n)+\mathcal{O}\big(\epsilon^3(n)\big)\Big)d+\ln{n}(1+\epsilon(n))}\nonumber\\&=\mathcal{O}\bigg(\frac{n^{1-\mu\epsilon(n)\ln^{-0.5+\epsilon}n+
\mathcal{O}\Big(\epsilon^2(n)\Big)}}{\epsilon(n)}\bigg).
\end{align}
Hence, by choosing $\epsilon(n)=\delta\ln^{-0.5+\epsilon}n$ for some constant $\delta$ (that we choose small enough) we get from the last inequality in (\ref{cannes5}) that for some constant $B>0$ and any constant $k$,
\begin{align}\label{cannes7}S_1=\mathcal{O}\bigg(ne^{-B\ln^{2\epsilon}n}\bigg)=\mathcal O \Big(\frac {n}{\ln^{k+1}{n}}\Big).
\end{align}

We argue similarly for the vertices $v$, $d(v)\leq \mu^{-1}\ln{n}-\ln^{0.5+\epsilon}{n}$.
In (\ref{cannes2}) let $t=\frac{1-\epsilon(n)}{2}$  where $\epsilon(n)=\delta\ln^{-0.5+\epsilon}n$ for some constant $\delta$ as above.
In analogy with (\ref{cannes9}) an upper bound for the expected number of vertices $v$ with $d(v)\leq \mu^{-1}\ln{n}-\ln^{0.5+\epsilon}{n}$ is
 \begin{align}\label{cannes10}S_2:=\sum_{d=0}^{\lfloor\mu^{-1}\ln{n}-\ln^{0.5+\epsilon}{n}\rfloor}b^d\mathbf{E}\big(n_v^{\frac{1-\epsilon(n)}{2}}(n_v-1)^{\frac{1-\epsilon(n)}{2}}\big).
\end{align}

We use similar calculations as in (\ref{cannes3})--(\ref{cannes7}) to show that
\begin{align}\label{cannes6}\sum_{d=0}^{\lfloor\mu^{-1}\ln{n}-\ln^{0.5+\epsilon}{n}\rfloor}b^d\big(\mathbf{E}(W_j^{1-\epsilon(n)})\big)^dn^{1-\epsilon(n)}=\mathcal{O}\bigg(ne^{-B\ln^{2\epsilon}n}\bigg).
\end{align}
This implies in analogy with (\ref{cannes9})--(\ref{cannes3}) that for some constant $B$ and any constant $k$,
\begin{align}\label{cannes11}S_2=\mathcal{O}\big(ne^{-B\ln^{2\epsilon}n}\big)=\mathcal O\Big(\frac {n}{\ln^{k+1}{n}}\Big).
\end{align}
Hence, if $s_1=0$ the number of bad vertices is bounded by $\mathcal O_{L^{1}} \Big(\frac {n}{\ln^{k+1}{n}}\Big)$, for any constant $k$. 
Thus, it follows from our previous explanation that the number of bad vertices for arbitrary $s_1\geq 0$ is bounded by $\mathcal O_{L^{1}} \Big(\frac {n}{\ln^{k}{n}}\Big)$.
\end{proof}

\begin{rem}
 We note from $(\ref{cannes5})$, $(\ref{cannes7})$ and $(\ref{cannes11})$ that we in fact get a sharper bound for the number of bad nodes, i.e., $\mathcal{O}\big(ne^{-B'\ln^{2\epsilon}n}\big)$ for some constant $B'>0$.
\end{rem}

\begin{rem}\label{rem1} From the calculations in the proof of Theorem \ref{lemma1} in particular in
$(\ref{cannes5})$, we see that we can get a much smaller error term for larger depths, i.e., for any constant $r$ there is a constant $C>0$ so that the number of nodes with $d(v)\geq C \ln n$  is bounded by $\mathcal{O}_{L^{1}}\big(\frac{1}{n^r}\big)$.
\end{rem}

\subsection{Proof of Theorem \ref{lemma}}\label{3}

\begin{proof}
We write 
\begin{align}\label{Z_n}Z_n:=\frac{D_n-\mu^{-1} \ln n}{\sqrt{\ln n}}.
\end{align}
By a classical result in probability theory, see e.g. \cite[Theorem 5.5.4]{gut}, the limit law in (\ref{splitdepth:2}) implies that (\ref{expsplitdepth2}) holds if 
$Z_n$ is uniformly integrable.
In particular this is true if $Z_n^2$ is uniformly integrable. This uniformly integrability also gives
\begin{align}\label{varsplitdepth3}\mathbf{E}\big(Z_n^2\big):=\frac{\mathbf{E}\Big(\big(D_n-\mu^{-1} \ln n\big)^2\Big)}{\ln n}\rightarrow\mathbf{E}\Big(N\big(0,\sigma ^2\mu^{-3}\big)^2\Big)=\sigma ^2\mu^{-3}.
\end{align}
Furthermore, the convergence results in (\ref{expsplitdepth2}) and (\ref{varsplitdepth3}) imply (\ref{varsplitdepth2}) for $k=n$.
By using the same coupling argument as in (\ref{expsplitdepth1}) it is easy to show that the convergence result of the expected depth in (\ref{expsplitdepth2}) implies the convergence result of the expected average depth in (\ref{expsplitdepth3}).

Thus, it remains to show that $Z_n^2$ is uniformly integrable and that (\ref{varsplitdepth2}) for $k=n$ implies that (\ref{varsplitdepth2}) also holds for $\frac{n}{\ln n}\leq k <  n$.
 By a standard argument, see e.g.\ \cite[Theorem 5.5.4]{gut}, $Z_n^2$ is uniformly integrable if for some $p>1$ and $n_0$ large enough,
\begin{align}\label{hoj}\sup_{n>n_0}\mathbf{E}\big(|Z_n^{2}|^p\big):&=\sup_{n>n_0}\mathbf{E} \bigg(\bigg{|}\frac{(D_n-\mu^{-1} \ln n\big)^2}{\ln n}\bigg{|}^{p}\bigg),
 \end{align}
  is uniformly bounded. We choose $p=\frac{3}{2}$.
We show that this is true by using similar calculations as Devroye used in \cite{devroye3} for proving
the limit law of $D_n$ in (\ref{splitdepth:2}).
First, consider an infinite random path $u_1,u_2,\dots, $ in the skeleton tree $S_b$, where $u_1$ is the root.
Given $u_1$ and the split vector $\mathcal{V}_{u_i}=(V_1,\dots,V_b)$ for $u_i$, then $u_{i+1}$ is the $j$-th child of $i$ 
with probability $V_j$. Construct a random split tree with $n$ balls and let $u*$ be the unique leaf in the infinite path. 
Then by using a natural coupling, letting the $n$:th ball follow the random path, $D_n$ is in stochastic sense less than or equal to the distance between $u*$ and the root. In the coupling $D_n$ is less than this distance, if the $n$-th ball is sent to a leaf which splits and does not send this ball to one of its children (i.e, the $n$-th ball is one of the $s_0$ balls). If the $n$-th ball is one of the $s_1$ balls it is added to a child of $p(u*)$ (the parent of $u*$), i.e., it ends up at the same depth as $u*$.  Recall that $H_n$ denotes the height of a split tree with $n$ balls.
For all $\beta>0$ we have
\begin{align}\label{larger event}
 \mathbf{P}\Big(D_n>k+\beta\Big) \leq \mathbf{P}\Big(n(u_k)>\beta\Big)+\mathbf{P}\Big(H_{\beta}>\beta\Big),
\end{align}
and 
\begin{align}\label{larger event2}
\mathbf{P}\Big(D_n<k\Big)\leq \mathbf{P}\Big(n(u_k)\leq s+1\Big).
\end{align}
Recall that $\Delta=V_S$, where given $(V_1,\dots,V_b)$, $S=j$ with probability $V_j$. 
Then $n(u_k)$ is stochastically bounded by 
\begin{align}\label{binomialsigma} n(u_k) &\leq \mathrm{mBin}(n,\prod_{j=1}^{k}\Delta_j)+\mathrm{mBin}(s_1,\prod_{j=2}^{k}\Delta_j)\nonumber\\&+\mathrm{mBin}(s_1,\prod_{j=3}^{k}\Delta_j)+\dots + \mathrm{mBin}(s_1,\Delta_k)+s_1,
\end{align}
where $\Delta_j$ are i.i.d random variables distributed as $\Delta$.

Consider the probability 
$\mathbf{P}(D_n>k+\beta)$, where $k=\lfloor \mu^{-1} \ln n+\frac{x}{2}\sqrt{\ln n} \rfloor$ for $x \in R^{+}$. We bound this by bounding the probabilities in the right hand-side of (\ref{larger event}), choosing  $\beta=\lfloor\frac{x}{2}\ln^{0.2}(n)\rfloor$. First note that the bound of $n(u_k)$ in (\ref{binomialsigma}) implies that in stochastic sense
\begin{align}\label{binomialsigma2}n(u_k) &\leq \mathrm{mBin}\big(n,\prod_{j=1}^{k}\Delta_j\big)+
\mathrm{mBin}\big(s_1,\prod_{j=1}^{k}\Delta_j\big) +\dots \nonumber\\&~~~~+\mathrm{mBin}\big(s_1,\prod_{j=k-\lfloor\frac{x}{2}\ln^{0.1}n\rfloor+1}^{k}\Delta_j\big)+\lfloor\frac{s_1x}{2}\ln ^{0.1}n\rfloor.
\end{align}
Thus, we can bound the first probability in the right hand-side of  (\ref{larger event}) by
\begin{align}\label{binomialsigma3} &
 \mathbf{P}(n(u_k)>\beta)\leq  \mathbf{P}\bigg(\mathrm{mBin}\big(n,\prod_{j=1}^{k}\Delta_j\big)+\lfloor\frac{s_1x}{2}\ln ^{0.1}n\rfloor \geq\beta-1\bigg)\nonumber\\&+\mathbf{P}\bigg(\mathrm{mBin}\big(s_1,\prod_{j=1}^{k}\Delta_j\big)+\dots +\mathrm{mBin}\big(s_1,\prod_{j=k-\lfloor\frac{x}{2}\ln ^{0.1}n\rfloor}^{k}\Delta_j\big) >1\bigg).
\end{align}
For bounding the first probability in the right hand-side of the inequality in (\ref{binomialsigma3}), we use \cite[Lemma 4]{devroye3} which states a general result for bounding tail probabilities for mixed binomial $(m,Z)$ distributions where $Z$ is a random variable, thus we obtain 
\begin{align}\label{limitlaw}&
 P_1:=\mathbf{P}\bigg(\mathrm{mBin}\big(n,\prod_{j=1}^{k}\Delta_j\big)>\beta-\lfloor\frac{s_1x}{2}\ln ^{0.1}n\rfloor-1\bigg)\nonumber\\&\leq 
\mathbf{P}\bigg(\sum_{j=1}^{k}\ln \Delta_j>\ln\Big(\frac{\beta-\lfloor\frac{s_1x}{2}\ln ^{0.1}n\rfloor-1}{2n}\Big)\bigg)+\Big(\frac{e}{4}\Big)^\frac{\beta-\lfloor\frac{s_1x}{2}\ln ^{0.1}n\rfloor-1}{2}.
\end{align}
From (\ref{limitlaw}) we deduce that for $n$ large enough
\begin{align}
\label{limitlaw1} P_1&\leq \mathbf{P}\bigg(\dfrac{\sum_{j=1}^{k}\ln \Delta_j+k\mu}{\sqrt{k\sigma^2}}>
\dfrac{\ln\Big(\frac{\beta-\lfloor\frac{s_1x}{2}\ln ^{0.1}n\rfloor-1}{2n}\Big)+k\mu}{\sqrt{k\sigma^2}}\bigg)+\Big(\frac{e}{4}\Big)^\frac{\beta-\lfloor\frac{s_1x}{2}\ln ^{0.1}n\rfloor-1}{2}
\nonumber\\&~~~\leq \mathbf{P}\bigg(\dfrac{\sum_{j=1}^{k}\ln \Delta_j+k\mu}{\sqrt{k\sigma^2}}>
\dfrac{x \mu^{\frac{3}{2}}}{3\sigma}\bigg)+\Big(\frac{e}{4}\Big)^\frac{\beta-\lfloor\frac{s_1x}{2}\ln ^{0.1}n\rfloor-1}{2}.
\end{align}
Recall the notations $c:=\mathbf{E}\big(\Delta\big)$, $\mu=:\mathbf{E}\big(-\ln \Delta\big)$ and 
$\sigma ^2 =:\mathbf{Var}\big(\ln \Delta\big)$. Note that $c<1$. Since the $\Delta_j$, $j\in\{1,\dots,k\}$, are i.i.d random variables we can use the Marcinkiewicz-Zygmund inequalities, see e.g.\ \cite[Corollary 3.8.2]{gut}, which gives for $q\geq 2$,
\begin{align}\label{normal}
 \mathbf{E}\bigg(\bigg{|}\sum_{j=1}^{k}\ln \Delta_j+k\mu \bigg{|}^q\bigg)\leq B_qk^{\frac{q}{2}}\mathbf{E}\bigg(\Big{|}\ln \Delta_j+\mu \Big{|}^q\bigg),
\end{align}
where $B_q$ is a constant only depending on $q$.
By using the Markov inequality and (\ref{normal}) we get from (\ref{limitlaw1}) that for $n$ large enough
\begin{align}\label{limitlaw2}&P_1\leq\dfrac{\mathbf{E}\bigg(\bigg(\dfrac{\sum_{j=1}^{k}\ln \Delta_j+k\mu}{\sqrt{k\sigma^2}}\bigg)^4\bigg)}{\Bigg(\dfrac{x \mu^{\frac{3}{2}}}{3\sigma}\Bigg)^4}+\Big(\frac{e}{4}\Big)^\frac{\beta-\lfloor\frac{s_1x}{2}\ln ^{0.1}n\rfloor-1}{2}\nonumber\\& \leq 
\dfrac{B_4\mathbf{E}\bigg(\Big{|}\ln \Delta_j+\mu \Big{|}^4\bigg)}{\Bigg(\dfrac{x \mu^{\frac{3}{2}}}{3}\Bigg)^4}+\Big(\frac{e}{4}\Big)^\frac{\beta-\lfloor\frac{s_1x}{2}\ln ^{0.1}n\rfloor-1}{2}=\frac{C}{x^4}+\Big(\frac{e}{4}\Big)^\frac{\beta-\lfloor\frac{s_1x}{2}\ln ^{0.1}n\rfloor-1}{2},
\end{align}
for the constant 
$C=\frac{B_4\mathbf{E}(|\ln \Delta_j+\mu |^4)3^4}{\mu^{6}}<\infty$ (recall from section \ref{not} that 
all moments of $|\ln \Delta|$ are bounded).
The Markov inequality  implies that
\begin{align}\label{binomialsigma4}
 &\mathbf{P}\bigg(\mathrm{mBin}\big(s_1,\prod_{j=1}^{k}\Delta_j\big)+\dots+ \mathrm{mBin}\big(s_1,\prod_{j=k-\lfloor\frac{x}{2}\ln ^{0.1}n\rfloor+1}^{k}\Delta_j\big)  \geq 1\bigg)\nonumber\\&\leq \mathbf{E}\bigg(\mathrm{mBin}\big(s_1,\prod_{j=1}^{k}\Delta_j\big)+\dots+ \mathrm{mBin}\big(s_1,\prod_{j=k-\lfloor\frac{x}{2}\ln ^{0.1}n\rfloor+1}^{k}\Delta_j\big)\bigg)\nonumber\\&=\mathcal{O}\Big(c^{\lfloor \frac{x}{2}\ln ^{0.1}n\rfloor}\Big),~~~~\mathrm{for}~~c:=\mathbf{E}(\Delta)<1.
\end{align}

We now consider the other probability i.e., $\mathbf{P}\Big(H_{\beta}>\beta\Big)$. (Note that this probability is 0 if $s_0>0$ or $s_1>0$.)
By applying (\ref{upperboundprofile}) we get
\begin{align}\label{mittag1} \mathbf{P}\Big(H_{\beta}>\beta\Big)\leq b^{\beta}\mathbf{P}\Big(n(v)\geq 2\Big),
\end{align} where $v$ is a vertex at depth $\beta-1$.
From (\ref{cannes}) we deduce for $t=0.75$,
\begin{align}\label{cannes,1}\mathbf{P}(n_v\geq 2)\leq  {\mathbf{E}(n_v^{0.75}(n_v-1)^{0.75})}.
\end{align}
Let $X_{\beta}$ be a mixed binomial $(n,\prod_{j=1}^{\beta}W_j)$, where $W_j,~j\in\{1,\dots,\beta\}$ are i.i.d.\ random variables distributed as $V$.
Note similarly as in (\ref{cannes8}) that (\ref{cannes,1}) is bounded by the expectation of
\begin{align}\label{cannes8,1} \mathbf{E}\Big(X_{\beta}^{0.75}(X_{\beta}-1)^{0.75} \big{|} \mathscr{G}_{\beta}\Big). 
\end{align}
We note similarly as in (\ref{binomial4}) that the Lyapounov inequality gives 
\begin{align}\label{binomial4,1}\mathbf{E}\Big(X_{\beta}^{0.75}(X_{\beta}-1)^{0.75}\big{|} \mathscr{G}_{\beta}\Big)
&\leq (\beta^2-\beta)^{0.75}\prod_{j=1}^{\beta}W_j^{1.5}\leq \Big(\beta\prod_{j=1}^{\beta}W_j\Big)^{1.5}.
\end{align}
Again the fact that $\mathbf{E}(W_j^2)<\mathbf{E}(W_j)=\frac{1}{b}$ (since $W_j\in (0,1)$),
gives that there is a $\delta>0$ such that
\begin{align}\label{height} \mathbf{P}\bigg(H_{\beta}>\beta\bigg)\leq b^{-\delta\beta}\beta^{1.5}.
\end{align}
We now consider the probability 
$\mathbf{P}(D_n<k)$, where $k=\lfloor \mu^{-1} \ln n-x\sqrt{\ln n} \rfloor$ for $x \in R^{+}$, and use the bound of the larger probability in (\ref{larger event2}).
We have 
\begin{align}\label{limitlaw3} \mathbf{P}(n(u_k)\leq s+1)\leq  \mathbf{P}\bigg(-ks+\mathrm{Bin}(n,\prod_{j=1}^{k}\Delta_j)\leq s+1\bigg).
 \end{align}
Again by applying \cite[Lemma 4]{devroye3} and using similar calculations as in (\ref{limitlaw})--(\ref{limitlaw2}), we get for $n$ large enough \begin{align}\label{limitlaw4}P_2&\leq   
\mathbf{P}\bigg(\dfrac{\sum_{j=1}^{k}\ln \Delta_j+k\mu}{\sqrt{k\sigma^2}}<
\dfrac{\ln\Big(\frac{2(s(k+1)+1)}{n}\Big)+k\mu}{\sqrt{k\sigma^2}}\bigg)+\Big(\frac{2}{e}\Big)^{s(k+1)+1}
\nonumber\\\leq &\mathbf{P}\bigg(\dfrac{\sum_{j=1}^{k}\ln \Delta_j+k\mu}{\sqrt{k\sigma^2}}<
\dfrac{x \mu^{\frac{3}{2}}}{3\sigma}\bigg)+\Big(\frac{2}{e}\Big)^{s(k+1)+1}
\nonumber\\&\leq
\dfrac{B_4\mathbf{E}\bigg(\Big{|}\ln \Delta_j+\mu \Big{|}^4\bigg)}{\Bigg(\dfrac{x \mu^{\frac{3}{2}}}{3}\Bigg)^4}+\Big(\frac{2}{e}\Big)^{s(k+1)+1}=C\frac{1}{x^4}+\Big(\frac{2}{e}\Big)^{s(k+1)+1},
\end{align}
for the constant 
$C=\frac{B_4\mathbf{E}(|\ln \Delta_j+\mu |^4)3^4}{\mu^{6}}<\infty$. 
Now we can show that for $n_0$ large enough $\sup_{n>n_0}\mathbf{E}\big(|Z_n^{2}|^{\frac{3}{2}}\big)$ in (\ref{hoj}) is uniformly bounded:
By the choice of $k$ and $\beta$, we get from (\ref{limitlaw2}), (\ref{binomialsigma4}), (\ref{height}) and (\ref{limitlaw4}) that for for $n_0$ large enough
\begin{align}\label{uni}\sup_{n>n_0}\mathbf{E}\big(|Z_n^{2}|^{\frac{3}{2}}\big):&=\sup_{n>n_0}\mathbf{E} \bigg(\bigg{|}\frac{(D_n-\mu^{-1} \ln n\big)^2}{\ln n}\bigg{|}^{\frac{3}{2}}\bigg)
\nonumber\\&=\sup_{n>n_0}\int_{x=0}^{\infty} 3x^2\mathbf{P} \bigg(\bigg{|}\frac{D_n-\mu^{-1} \ln n}{\sqrt{\ln n}}\bigg{|}>x\bigg)dx \nonumber\\&\leq \sup_{n>n_0} \bigg{\{}\int_{x=1}^{\infty}\bigg(\frac{6C}{x^2}+3x^2\Big(\frac{e}{4}\Big)^\frac{\beta-\lfloor\frac{s_1x}{2}\ln ^{0.1}n\rfloor-1}{2}+3x^2\Big(\frac{2}{e}\Big)^{s(k+1)+1}\nonumber\\&+\mathcal{O}\Big(x^2c^{\lfloor \frac{x}{2}\ln ^{0.1}n\rfloor}\Big)+ 3 x^2b^{-\delta\beta}\beta^{1.5}
\bigg)dx\bigg{\}}+1<\infty,\end{align}
 and thus $Z_n^2$ is uniformly integrable so that (\ref{varsplitdepth3}) holds, which shows (\ref{varsplitdepth2}) for $k=n$.

From this result it is now easy to show as we explain below that (\ref{varsplitdepth2}) also holds for all $k$, $\frac{n}{\ln n}\leq k <  n$. 
Recall that we denote the depth of ball $k$, when it is added to the tree by $D_{k}^{f}$. 
As we argued for proving (\ref{expsplitdepth1}),  in stochastic sense  for $k\leq n$,
\begin{align}\label{begr}
D_{k}^{f}\leq D_k\leq D_n.
\end{align}
From (\ref{splitdepth:2}) it follows that for all $\frac{n}{\ln n}\leq k\leq n$,
\begin{align*} \frac{D_{k}^{f}-\mu^{-1} \ln n}{\sqrt{\sigma ^2\mu^{-3}\ln n}}\stackrel{d}\rightarrow N(0,1).
\end{align*}
By using this and (\ref{begr}), for $\frac{n}{\ln n}\leq k\leq n$,
\begin{align*} \frac{D_{k}-\mu^{-1} \ln n}{\sqrt{\sigma ^2\mu^{-3}\ln n}}\stackrel{d}\rightarrow N(0,1).
\end{align*}
We need to show that for $\frac{n}{\ln n}\leq k\leq n$,
\begin{align}\label{axel}\frac{\mathbf{E}\Big(\big(D_k-\mu^{-1} \ln n\big)^2\Big)}{\ln n}\rightarrow\mathbf{E}\Big(N\big(0,\sigma ^2\mu^{-3}\big)^2\Big).
\end{align}
As for $D_n$ this follows if for $n_0$ large enough,
\begin{align}\label{uni2}&\sup_{n>n_0}\mathbf{E} \bigg(\bigg{|}\frac{\big(D_k-\mu^{-1} \ln n\big)^2}{\ln n}\bigg{|}^\frac{3}{2}\bigg)
<\infty.
\end{align}
We have for $k\leq n$,
\begin{align*} \mathbf{P} \bigg(\frac{D_{k}^{f}-\mu^{-1} \ln n}{\sqrt{\ln n}}\geq x\bigg)&\leq \mathbf{P} \bigg(\frac{D_k-\mu^{-1} \ln n}{\sqrt{\ln n}}\geq x\bigg)\\&\leq\mathbf{P} \bigg(\frac{D_n-\mu^{-1} \ln n}{\sqrt{\ln n}}\geq x\bigg),
 \end{align*}
and
\begin{align*} \mathbf{P} \bigg(\frac{D_n-\mu^{-1} \ln n}{\sqrt{\ln n}}<x\bigg)&\leq \mathbf{P} \bigg(\frac{D_k-\mu^{-1} \ln n}{\sqrt{\ln n}}<x\bigg)\\&\leq\mathbf{P} \bigg(\frac{D_{k}^{f}-\mu^{-1} \ln n}{\sqrt{\ln n}}<x\bigg).
 \end{align*}
Thus, (\ref{uni2}) follows from the calculations in (\ref{uni}).
This shows that (\ref{varsplitdepth2}) holds for all $k$, $\frac{n}{\ln n}\leq k <  n$, follows from the fact that (\ref{varsplitdepth2}) holds for $k = n$. 
\end{proof}

We now prove the two corollaries of Theorem \ref{lemma}.
\begin{proof}[Proof of Corollary \ref{cor1}]
We show (\ref{branoder2}) and from this it is obvious that (\ref{branoder22}) also holds since Theorem \ref{lemma1} implies that the bad vertices are few enough so that we could equally  sum over all vertices.

First note that (\ref{axel}) gives that for the balls  $\lfloor\frac{n}{\ln n}\rfloor\leq k\leq n$,
\begin{align}\label{branoder}
 \mathbf{E}(D_k-\mu^{-1}\ln{n})^2=  \mu^{-3}\sigma^2\ln n+o\big(\ln n\big).
\end{align} 

Recall that a vertex $v$ in a split tree $T$ is called good if 
\begin{align*}\mu^{-1}\ln{n}-\ln^{0.5+\epsilon}{n}\leq d(v)\leq \mu^{-1}\ln{n}+\ln^{0.5+\epsilon}{n},
\end{align*}
and that we write $\mathsf{V}^{\ast}\big(T^{n}\big)$ for the set of good vertices in $T^{n}$ and $N^{*}=|\mathsf{V}^{\ast}\big(T^{n}\big)|$ for the number of good vertices.
Note that (\ref{vertexball}) and Theorem \ref{lemma1} implies that
\begin{align}\label{sushi}
 \mathbf{E}(N^{\ast})=  \alpha n+o\big(n\big).
\end{align}

We will now consider subtrees defined similarly  as the $T_{r,B}$, $r\in R$, subtrees we used in the proof of Theorem \ref{assumption1}. However, instead of using the product $M_v^N$ for defining the stopping time in each branch we use the real subtree size $n_v$: Let $U$ be the set of vertices such that $u\in U$, if and only if $n_u\leq \ln^{0.4} n$ and $n_{p(u)}>\ln^{0.4} n$ (where $p(u)$ is the parent of $u$), and consider  all subtrees $T_u^{\ln ^{0.4}n},~u\in U$, rooted at  $u$.

It is an immediate consequence of Lemma \ref{lem11} that the first and second moment of the height of a subtree with $\ln^{0.4} n$ balls is bounded by $\mathcal{O}\Big(\ln ^{0.4}n\Big)$ and $\mathcal{O}\Big(\ln ^{0.8}n\Big) $, respectively. However, there are much stronger bounds, e.g., \cite{devroye3} since split trees are of logarithmic order.
Hence, since the subtrees are small by applying that $d(v)-\mu^{-1}\ln{n}=(d(u)-\mu^{-1}\ln{n})+(d(v)-d(u))$ and summing over all good vertices  we get
\begin{align}\label{branoder3}&
\mathbf{E} \Big(\sum_{u}\sum_{v\in \mathsf{V}^{\ast}\big(T_u^{\ln ^{0.4}n}\big)}(d(v)-\mu^{-1}\ln{n})^2\Big)\nonumber\\&=\mathbf{E}\Big(\sum_{u}\sum_{v\in \mathsf{V}^{\ast}\big(T_u^{\ln ^{0.4}n}\big)}(d(u)-\mu^{-1}\ln{n})^2\Big)+o(n\ln n).
\end{align}
In (\ref{branoder3}) we use the bound for the good vertices, but it is obvious from Theorem \ref{lemma1} that the bad vertices are few enough so that one could equally sum over all vertices.
The number of subtrees that hold the balls $k<\lfloor\frac{n}{\ln n}\rfloor$ is trivially bounded by $\lfloor\frac{n}{\ln n}\rfloor$. Thus, the number of nodes in these subtrees is bounded by 
$\mathcal{O}_{L^1}\Big(\frac{n}{\ln^{0.6}n}\Big)$. Let ${N}_u^{\ast}=|\mathsf{V}^{\ast}\big(T_u^{\ln ^{0.4}n}\big)|$ be the number of good vertices in $T_u^{\ln ^{0.4}n}$.
Hence, by applying that the subtrees $T_u^{\ln ^{0.4}n},~u\in U$, are small so that $(d(v)-\mu^{-1}\ln{n})^2$ do not differ more than $\mathcal{O}\Big(\ln ^{0.8}n\Big)$ for different vertices $v\in T_u^{\ln ^{0.4}n}$, together with (\ref{sushi}) we get
\begin{align}\label{branoder4}&\mathbf{E}\Big(\sum_{u}\sum_{v\in \mathsf{V}^{\ast}\big(T_u^{\ln ^{0.4}n}\big)}(d(u)-\mu^{-1}\ln{n})^2|n_u\Big)\nonumber\\&=\sum_{u}\mathbf{E}\Big(\sum_{v\in \mathsf{V}^{\ast}\big(T_u^{\ln ^{0.4}n}\big)}\sum_{k\in T_u}\frac{(D_k-\mu^{-1}\ln{n})^2}{n_u}|n_u\Big)+o_{L^1}\big(n \ln n\big)
\nonumber\\&=\sum_{u}\sum_{k\in T_u}\frac{\mathbf{E}\Big((D_k-\mu^{-1}\ln{n})^2|n_u\Big)}{n_u}(\alpha n_u+o(n_u))+
o_{L^1}\big(n \ln n\big)\end{align}

Recall from (\ref{haj1}) in Lemma \ref{lem16} that
\begin{align*}\mathbf{E}(N)=\mathbf{E}\Big(\sum_{r\in R} N_{r}I\{|n_{r}-M_{r}^{n}|\leq B^{0.6}\}\Big)+\mathcal{O}\Big(\frac{n}{B^{0.1}}\Big),\end{align*} where 
$M_{r}^{n}$ by definition is less than $B$. 
If we choose $B:=\ln ^{0.3}n$ this means that for the expected value in the right hand-side we can assume that 
$n_{r}\leq \ln ^{0.4}n$. Hence, the expected number of (good) vertices in $T^{n}$ that are not in the subtrees $T_u^{\ln ^{0.4}n},~u\in U$, is bounded by $\mathcal{O}\Big(\frac{n}{B^{0.1}}\Big)$ for $B=\ln ^{0.3}n$.
Hence, this bound  implies that the expected  value of the last equality in (\ref{branoder4}) is equal to 
\begin{align*}\sum_{k\in T^n}\mathbf{E}\Big((D_k-\mu^{-1}\ln{n})^2\Big)(\alpha+o(1))+o\big(n \ln n\big)\nonumber\\=
 \Big(\mu^{-3}\sigma^2\ln n+o\big(\ln n\big)\Big)(\alpha n+o(n))+o\big(n \ln n\big).
\end{align*}
\end{proof}

\begin{proof}[Proof of Corollary \ref{cor2}]
 As in Corollary \ref{cor1} we only show the result for the good vertices, i.e., (\ref{indsum3}). From the proof it is obvious that also (\ref{indsum33})  holds by applying Theorem \ref{lemma1}, showing that the number of bad vertices is covered by the error term.
We observe the obvious fact that the sum of those $n_i,~i\in\{1,\dots, b^L\}$,  
which are less than $\frac{n}{b^{kL}}$ for large enough $k$, is bounded by
\begin{align}\label{boundy}b^L \cdot\frac{n}{b^{kL}}=\mathcal{O}\Big(\frac{n}{\ln^4 n}\Big).
\end{align}
(Note that by choosing $k$ large enough in (\ref{boundy}), 
the power of the logarithm can be arbitrarily large.)

Recall that $\mathsf{V}^{\ast}\big(T_i\big)$ is the set of good vertices in $T_i$
and that $\Omega_L$  is the $\sigma$-field generated by $\{n_v,~d(v)\leq L\}$. Let \begin{align*}Z_i:=\sum_{v \in \mathsf{V}^{\ast}\big(T_i\big)}\frac {{(d_i(v)-\mu^{-1}\ln{n_i})^2}}{ \mu^{-3}\ln^3{n_i}}.\end{align*} 
Thus, from (\ref{branoder2}) it follows that 
\begin{displaymath}\sum_{i=1}^{b^L}\mathbf{E} \bigg(Z_i\Big|\Omega_L \bigg)=\sum_{i=1}^{b^L}\frac {\sigma^2 \alpha n_i}{\ln^2{n_i}}
+\sum_{i=1}^{b^L}\frac {o(n_i)}{\ln^{2}{n_i}}.\end{displaymath}

Let $k>0$ be a fixed constant and assume that $n_i$ is at least $\frac{n}{b^{kL}}$; by Taylor expansion we  get
\begin{align}\label{hejsan}\frac{1}{\ln^2 n_i}=\frac{1}{\ln^2 n}+\mathcal{O}\Big(\frac{\ln\ln n}{\ln^3 n}\Big).
\end{align}
By applying (\ref{boundy}) for $k$ large enough to cover those $n_i$ that are less than $\frac{n}{b^{kL}}$ in an error term $o\big(\frac{n}{\ln^2 n}\big)$, and using (\ref{hejsan}) we deduce
\begin{align}\label{hals}\sum_{i=1}^{b^L}\frac {o(n_i)}{\ln^{2}{n_i}}=o\left(\sum_{i=1}^{b^L}\frac {n_i}{\ln^{2}{n_i}}\right)+o\big(\frac{n}{\ln^2 n}\big)=o\big(\frac{n}{\ln^2 n}\big).
\end{align}

 Hence, since we can assume that $n_i$ is at least $\frac{n}{b^{kL}}$ for large enough $k$, by Taylor expansion
\begin{align}\label{pontus1}\sum_{i=1}^{b^L}\frac {\sigma^2\alpha n_i}{\ln^2{n_i}}
=\frac{\sigma^2\alpha n}{\ln^2{n}}+\mathcal O\Big(\sum_{i=1}^{b^L}\frac {n_i\ln\ln{n}}{\ln^3{n}}\Big)=
\frac{\sigma^2\alpha n}{\ln^2{n}}+o\big(\frac{n}{\ln^2 n}\big).\end{align}

Since only the good vertices are considered, and the random variables  conditioned on  $\Omega_L $
 are independent for $i\in\{1,\dots,b^L\}$,
\begin{align}\label{indsum}&\mathbf{Var}   \Big(\sum_{i=1}^{b^L}Z_i
\Big|\Omega_L \Big)=  \sum_{i=1}^{b^L}\mathbf{Var} \Big(Z_i  
\Big|\Omega_L \Big)
\leq \mu^{3}\sum_{i=1}^{b^L}\mathbf{Var} \Big(\sum_{v \in \mathsf{V}^{\ast}\big(T_i\big)} \frac{1}{\ln^{2-2\epsilon}{n_i}}\Big|\Omega_L \Big).
\end{align} 
Thus, the well-known Minkowski's inequality and the fact that $\mathbf{E}(N^2)=\mathcal{O}(n^2)$ imply
\begin{align}\label{indsum2}&\mathbf{Var}   \Big(\sum_{i=1}^{b^L}Z_i
\Big|\Omega_L \Big)=\mathcal O\left(\sum_{i=1}^{b^L}\frac {n_i^2}{\ln^{4-4\epsilon}{n_i}}\right).\end{align} 
Similarly as in (\ref{subtreei}) for $\beta$ large enough,
\begin{align}\label{sum4}&\sum_{i=1}^{b^L}\mathbf{E}(n_i^2)=o(\frac{n^2}{\ln ^{4} n}), \end{align}
Applying (\ref{sum4}), Chebyshev's inequality gives (\ref{indsum3}). 

\end{proof}

\section {Results on the Total Path Lengths}
We complete this study with some results and a conjecture of the ``total path length'' random variables. Recall from
Section \ref{not} the definitions of the two types of total path length $\Psi(T)$ and $\Upsilon(T)$, i.e., the sum of the depths of balls and the sum of the depths of nodes, respectively.

From (\ref{expsplitdepth3})  we have
 \begin{align}\label{exptotal1}
 \mathbf{E}\Big(\Psi{(T^{n})}\Big)=\mu^{-1} n\ln n+nq(n),
\end{align}
where
$q(n)=o(\ln^{0.5} n)$ is a function that depends on the type of split tree.

Similarly, by using(\ref{vertexball}) in Theorem \ref{assumption1} and the profile result in Theorem \ref{lemma1} including Remark \ref{rem1} (which gives a smaller bound of the expected number of vertices with depths much bigger than the depths of the good vertices), we get 
 \begin{align}\label{exptotal2}
 \mathbf{E}\Big(\Upsilon{(T^{n})}\Big)=\mu^{-1}\alpha n\ln n+nr(n),
\end{align}
where $\alpha$ is the constant that occurs in (\ref{vertexball}) and 
$r(n)=o(\ln n)$ is a function that depends on the type of split tree. 

\begin{ass2}\label{christmas2}
Assume that the functions $q(n)$ in (\ref{exptotal1}) converges to some constant $\varsigma$.
\end{ass2} 
In \cite{Neininger} there is an analogous assumption.
Examples of split trees where it is shown that $q(n)$ converges to a constant  are binary search trees (e.g.\ \cite{jan3}), random $m$-ary search trees \cite{Mahmoud}, quad trees \cite{Neininger} and the random median of a $(2k+1)$-tree \cite{Roesler}, tries and Patricia tries \cite{bourdon}. 

\begin{ass3}\label{assumption,1} 
We assume that the result  in (\ref{vertexball}) in Theorem \ref{assumption1} can be improved to
\begin{align*}\mathbf{E}(N)=\alpha n+f(n),
\end{align*} where $f(n)=\mathcal{O}\Big(\frac{n}{\ln^{1+\epsilon} n}\Big)$.\end{ass3} 
 Stronger second order terms of the size have previously been shown to hold e.g., for $m$-ary search trees \cite{Mahmoud2}, for these $f(n)$ in assumption (A3) is $o(\sqrt{n})$ when $m\leq 26$ and is
$\mathcal{O}\Big(n^{1-\epsilon}\Big)$ when $m\geq 27$.  
 Further, as described in Section \ref{strong} tries are special cases of split trees which are not as random as other types of split trees. Flajolet and Vall\'ee (personal communication) have recently shown that also for most tries (as long as $-\ln V$ is not too close to being lattice) assumption (A3) holds.


\begin{thm}\label{christmas}
 Assume that (A1)--(A3)  hold, then also 
$r(n)$ converges to some constant $\zeta$.
\end{thm}
Let \begin{align}\label{jul} \Gamma_n:=\alpha nq(n)-nr(n),
\end{align}
and note that 
\begin{align}\label{snow}
 \alpha\mathbf{E}\Big(\Psi{(T^{n})}\Big)-\mathbf{E}\Big(\Upsilon{(T^{n})}\Big)=\Gamma_n.
\end{align}
For proving  Theorem \ref{christmas} we will show that $\frac{\Gamma_n}{n}$ converges to a constant. 
We write $\sumx_{v}$ for a sum where we sum over all vertices $v \in T^{n}$ except the root i.e., $v\neq \sigma$.
First we recall that the total pathlength for the balls is equivalent to the sum of all subtree sizes (except for the the whole tree) for the balls i.e.,
\begin{align}\label{equi1}
 \Psi{(T^{n}})=\sumx_{v} n_v,
\end{align}
where $\sigma$ is the root of $T^n$. Similarly we recall that  the total pathlength for the nodes is equivalent to the sum of all subtree sizes (except for the the whole tree) for the nodes i.e.,
\begin{align}\label{equi2}
 \Upsilon{(T^{n}})=\sumx_{v} N_v,
\end{align}
where $\sigma$ is the root of $T^n$.
Hence, by assuming (A3) we get from (\ref{snow}) that
  \begin{align}\label{snowy}
\Gamma_n&=\alpha\mathbf{E}\Big(\sumx_{v} n_v\Big)-\mathbf{E}\Big(\sumx_{v} \Big(\alpha n_v+\mathcal{O}\Big(\frac{n_v}{\ln^{1+\epsilon} n_v}\Big)\Big)\Big)\nonumber\\&=\mathbf{E}\Big(\sumx_{v}
\mathcal{O}\Big(\frac{n_v}{\log ^{1+\epsilon }n_v}\Big)\Big)
.
\end{align}
We will again consider the $T_{r,B},~r\in R$, subtrees from the proof of Theorem \ref{assumption1} in Section \ref{mainsection} (defined such that $M_{r}^{n}:=n\prod_{j=1}^{d(r)}W_j< B$ and $M_{p(r)}^{n}:=n\prod_{j=1}^{d(r)-1}W_j\geq B$). However, here we choose $B$ differently, i.e., $B=\epsilon^{-20}$.

\begin{Lemma} \label{lem7} 
Assume that (A1)--(A3)  hold, then 
\begin{align}\label{sumi2,2}
 \alpha\mathbf{E}\Big(\Psi{(T^{n})}\Big)-\mathbf{E}\Big(\Upsilon{(T^{n})}\Big)=\Gamma_n=\mathcal{O}\big(n\big).
\end{align}
Furthermore, 
\begin{align}\label{sumi2,1}\Gamma_n=\sum_{r\in R}\mathbf{E}\big(\Gamma_{n_r}\big)+o(n),
\end{align}
where \begin{align}\label{sumi23}\Gamma_{n_r}
=\alpha\mathbf{E}\Big(\Psi{(T_{r,B})}\Big|n_r\Big)-\mathbf{E}\Big(\Upsilon{(T_{r,B})}\Big|n_r\Big).
\end{align}
\end{Lemma}

\begin{proof} Assuming (A3),  we get from (\ref{snowy}) that
\begin{align}\label{sumi2,3}
 \Gamma_n&= \mathbf{E}\Big(\sumx_{v}
\mathcal{O}\Big(\frac{n_v}{\log ^{1+\epsilon }n_v}\Big)\Big)\nonumber\\&= \sum_k\mathbf{E}\Big(\sumx_{v:2^{k}\leq n_v<2^{k+1}}\mathcal{O}\Big(\frac{n_v}{\log ^{1+\epsilon }n_v}\Big)\Big)\nonumber\\&=
\sum_k\mathcal{O}\Big(\frac{n}{2^{k}}\cdot\frac{ 2^{k}}{ k^{1+\epsilon }}\Big)=\mathcal{O}\big(n\big),
\end{align} where we applied Corollary \ref{lem6,1} in the last equality.
In the same way, for the vertices which are not in the $T_{r,B},~r\in R$,
subtrees (ignoring the root $\sigma$) we deduce that

\begin{align} \label{idag}
 \mathbf{E}\Big(\sum_{\radsumma{v\notin T_{r,B},~r\in R,}{v\neq \sigma}}
\mathcal{O}\Big(\frac{n_v}{\log ^{1+\epsilon }n_v}\Big)\Big)=\sum_{k\geq \epsilon^{-20}}  \mathcal{O}\Big(\frac{n}{2^{k}}\cdot\frac{ 2^{k}}{ k^{1+\epsilon }}\Big)=o(n).
\end{align}

 Hence,  (\ref{sumi2,1}) follows from (\ref{idag}) and  (\ref{sumi2,3}).
\end{proof}

\begin{proof}[Proof of Theorem \ref{christmas}]
We will use the same type of proof as the proof of (\ref{vertexball}) in Theorem \ref{assumption1}.
We start with two arbitrary values  of  the cardinality $n$ and $\widehat{n}$, where $\widehat{n}\geq n$, and show that
\begin{align}\label{limsuphoj}\Big|\frac{\Gamma_n}{n}-\frac{\Gamma_{\widehat{n}}}{\widehat{n}}\Big|=\mathcal{O}\big(\epsilon\big)~~~~\mathrm{as}~n~\rightarrow\infty,~~\forall\epsilon>0.
\end{align}
Since (\ref{limsuphoj}) 
implies that $\frac{\Gamma_{ n}}{n}$ is Cauchy it also converges to some constant as $n$ tends to infinity; hence, we deduce Theorem  
\ref{christmas}. Recall from the proof of Theorem \ref{assumption1} that a main application for the proof is to use (\ref{obvious}) in Lemma \ref{lem11}. Here we use an analogous applications of (\ref{sumi2,3}) in Lemma \ref{lem7},
 i.e., $\Gamma_{n}=\mathcal{O}\big(n\big)$.


Recall that we prove Lemma \ref{lem16} by showing 
\begin{align}\label{hoaloa}&\mathbf{E}\Big(\sum_r n_{r}I\{|n_{r}-M_{r}^{n}|\geq B^{0.6}\}\Big)=\mathcal{O}\Big(\frac{n}{B^{0.1}}\Big),
 \end{align}
and then applying (\ref{obvious}) in Lemma \ref{lem11}. In the same way by using (\ref{hoaloa}) and (\ref{sumi2,2}) as well as (\ref{sumi2,1}) in Lemma \ref{lem7} we get that
\begin{align}\label{hajjul}\Gamma_n=\mathbf{E}\Big(\sum_{r\in R} \Gamma_{n_r}I\{|n_{r}-M_{r}^{n}|\leq B^{0.6}\}\Big)+o(n)+\mathcal{O}\Big(\frac{n}{B^{0.1}}\Big).
\end{align}

Recall that $R'\subseteq R$ is the set of vertices such that $r\in R'$ if
\begin{align} \label{epsilonigen}|n_{r}-M_{r}^{n}|\leq B^{0.6},
 \end{align} and that
 $ R''\subseteq R'$ is the set of vertices such that $r\in R''$ if 
$r\in R'$ and
\begin{align}\label{epsilon2igen} 
\epsilon B< M_{r}^{n}< B.
\end{align}
Lemma \ref{lem16}  shows that we only need to consider the vertices in $r\in R'$.
Similarly as in (\ref{limsupha}) we get
 \begin{align}\label{limsuphajul}\Gamma_n=\mathbf{E}\Big(\sum_{r\in R''} \Gamma_{n_r}\Big)+\mathcal{O}\big(\epsilon n\big)+\mathcal{O}\Big(\frac{n}{B^{0.1}}\Big).
\end{align}
Recall from the proof of Theorem \ref{assumption1} that we sub-divide the 
$T_{r,B},~r\in R$, subtrees into smaller classes, wherein the $M_{r}^{n}$, $r\in R$, in each class are close to each-other. 
 As before we choose  $\gamma=\epsilon^2$, and  let $ Z=\{B,B-\gamma B,B-2\gamma B,\dots,\epsilon B\}$, where  $\epsilon=\frac{1}{k}$ for some positive integer $k$.
Recall that we write $R_z\subseteq R,~z\in Z$, 
for the set of vertices $r\in R$, such that  $M_{r}^{n}\in [z-\gamma B,z)$ and $M_{p(r)}^{n}\geq B$.
Hence, (\ref{limsuphajul}) gives 
\begin{align}\label{limsupnujul}&\Gamma_n
=\mathbf{E}\Big(\sum_{z\in Z}\sum_{r\in R'\cap  R_{z}}\Gamma_{n_r}\Big)+
\mathcal{O}\Big(\epsilon n\Big)+\mathcal{O}\Big(\frac{n}{B^{0.1}}\Big).
\end{align}
To approximate the expected value in (\ref{limsupnujul}) we apply a lemma that is similar to Lemma \ref{lem14}.
\begin{Lemma}\label{lemmatot} Adding $K$ balls to a tree can at most have an influence on
$\mathbf{E}\Big(\Psi{(T^{n})}\Big)$ and  $\mathbf{E}\Big(\Upsilon{(T^{n})}\Big)$, respectively, by $\mathcal{O}(K\ln(n+K))$.
\end{Lemma} 
\begin{proof}
Adding one ball to a tree with $n$ balls the expected depth is $\mathcal{O}(\ln n)$ and the expected number of additional nodes is $\mathcal{O}(1)$. Note that $\mathcal{O}(1)$ nodes only have distances $\mathcal{O}(1)$ between each other.) Hence, when the $K$-th ball is added the expected depth is $\mathcal{O}(\ln(n+K))$. Since $K$ balls give an expectation of $\mathcal{O}(K)$ nodes the result holds for both $\mathbf{E}\Big(\Psi{(T^{n})}\Big)$ and $\mathbf{E}\Big(\Upsilon{(T^{n})}\Big)$. 
 \end{proof}

Let $r_z$ be an arbitrarily chosen node in $R'\cap R_z$, where $z\in Z$. Similarly 
as in (\ref{limsupnu.1}),
by using (\ref{epsilonigen}) and Lemma \ref{lemmatot}, from (\ref{limsupnujul}) we get 
\begin{align}\label{limsupnuigen}&\Gamma_n=\sum_{z\in Z}\mathbf{E}(|R'\cap R_z|)\Big(\Gamma_{n_{r_z}}+\mathcal{O}\big(\gamma B\ln B\big)\Big)+\mathcal{O}\big(\epsilon n\big)+\mathcal{O}\Big(\frac{n}{B^{0.1}}\Big).
\end{align}
Define $b_x$ in a tree with cardinality $\lfloor x \rfloor$ as $b_x:=\frac{\Gamma_{\lfloor x \rfloor}}{\lfloor x \rfloor}$, and
 note from Lemma \ref{lem7} that $b_x=\mathcal{O}\big(1\big)$.
Recall that $S=\{1, 1-\gamma,1-2\gamma,\dots ,\epsilon\}$, where $\gamma=\epsilon^2$.
 Recall from (\ref{residual4.1}) that for each choice of $\gamma$ and $\alpha \in S$, there is a $\sigma_{\gamma}$ such that for a constant $c_{\alpha}$ (depending on $\alpha$),
\begin{align}\label{annandag}\Big|\dfrac{\mathbf{E}(|R'\cap R_{\alpha B}|)}{\frac{n}{B}}-c_{\alpha} \Big|\leq \gamma^{2}+\mathcal{O}\Big(\frac{1}{B^{0.1}}\Big),
\end{align}
whenever
 $\frac{n}{B}\geq \frac{1}{\sigma_{\gamma}}$. By choosing $B=\epsilon^{-20}$  for $n$ large enough  $\frac{n}{B}\geq \frac{1}{\sigma_{\gamma}}$ so that (\ref{annandag}) holds.
Moreover,  since $\sum_{\alpha \in S}c_{\alpha}=\mathcal{O}(1)$ we have that $\sum_{\alpha \in S}c_{\alpha}\frac{\mathcal{O}(B\gamma\ln B )}{B}=\mathcal{O}(\gamma\ln \epsilon )$.
Recall that  $\gamma=\epsilon^2$.
Thus, for a constant $c_{\alpha}$ (depending on $\alpha$)  and $b_{\alpha B}=\mathcal{O}\big(1\big)$, (\ref{limsupnuigen}) and (\ref{annandag}) imply that
\begin{align}\label{limsuphej}\Gamma_n&=n\sum_{\alpha \in S}c_{\alpha}\frac{1}{B}\Big(b_{\alpha B} \alpha B+\mathcal{O}(B\gamma\ln B)\Big)+n\sum_{\alpha \in S}\mathcal{O}\big(b_{\alpha B}\gamma^{2}\big)+\mathcal{O}\big(\epsilon n\big)=
\nonumber\\&=n\sum_{\alpha \in S}\alpha b_{\alpha B}c_{\alpha}+\mathcal{O}\big(\epsilon n\big).
\end{align}
In analogy, also for $\widehat{n}\geq n$,
\begin{align}\label{limsuphejigen}&\Gamma_{\widehat{n}}=
\widehat{n}\sum_{\alpha \in S}\alpha b_{\alpha B}c_{\alpha}+\mathcal{O}\big(\epsilon\widehat{n}\big).
\end{align}
Thus, (\ref{limsuphoj}) follows, which shows Theorem \ref{christmas}.
\end{proof}

Finally we present a theorem that is applied in \cite{holmgren3}.
\begin{thm} \label{thm7} Let  $L=\lfloor \beta \log_b\ln{n}\rfloor$ for some large enough constant $\beta$. Assume that (A1)--(A3)  hold, then 
\begin{align}\label{TPL,1}&\sum_{i=1}^{b^L}\frac {\Psi{(T_i)}}{\mu^{-2}{\ln^2{n_i}}}=\sum_{i=1}^{b^L}
\frac{ n_i}{\mu^{-1}\ln{n_i}}+\frac{ n \varsigma }{\mu^{-2}\ln^2{n}} +o_p\Big(\frac {n}{{\ln^2{n}}}\Big),\end{align}
and
\begin{align}\label{TPL}&\sum_{i=1}^{b^L}\frac {\Upsilon{(T_i)}}{\mu^{-2}{\ln^2{n_i}}}=\sum_{i=1}^{b^L}
\frac{\alpha n_i}{\mu^{-1}\ln{n_i}}+\frac{ n \zeta }{\mu^{-2}\ln^2{n}} +o_p\Big(\frac {n}{{\ln^2{n}}}\Big).\end{align}
\end{thm}
\begin{proof}
We only show (\ref{TPL}), since we can use exactly the same type of arguments 
 for showing (\ref{TPL,1}).

First, (\ref{exptotal2}) gives
\begin{align}\label{exptotal1,1}
\mathbf{E}\bigg(\sum_{i=1}^{b^L}\frac {\Upsilon{(T_i)}}{\mu^{-2}{\ln^2{n_i}}}\Big|\Omega_L\bigg)=
\sum_{i=1}^{b^L}\frac {{\alpha n_i}}{\mu^{-1}{\ln{n_i}}}+\sum_{i=1}^{b^L}\frac{n_ir(n_i)}{\mu^{-2}{\ln^2{n_i}}}.\end{align} 

Note that conditioned on $\Omega_L$, the summands $\Upsilon{(T_i)},~i\in\{1,\dots,b^L\}$  are independent.  By applying  the Cauchy-Schwarz inequality, and using the facts that $\mathbf{E}(N^2)=\mathcal{O}\big(n^2\big)$ and that $\mathbf{E}\big(D_k^2\big)=\mathcal{O}\big(\ln^2 n\big)$ for all $k$, we deduce that
\begin{align}\label{sumtot}\mathbf{Var}\bigg(\sum_{i=1}^{b^L}
 {\Upsilon{(T_i)}}\Big|\Omega_L\bigg)&=\sum_{i=1}^{b^L}\mathbf{Var}\Big(
 {\Upsilon{(T_i)}}\Big|\Omega_L\Big)\nonumber\\\leq \sum_{i=1}^{b^L}\mathbf{E}\Big(
\Upsilon{(T_i)^2}\Big|\Omega_L\Big)&= \sum_{i=1}^{b^L}\mathcal O\Big(n_i^2\ln^2{n_i}\Big).\end{align}
 Similarly as in (\ref{sum4}), for any constant $k$ (and choosing the constant $\beta$ in $L$ large enough) the following holds
\begin{align}\label{sumi}
\mathbf{E}\bigg(\sum_{i=1}^{b^L}n_i^2\bigg)=\frac{n^2}{\ln ^kn}.\end{align} 
Thus, for a large enough constant $\beta$, by taking expectations in (\ref{sumtot}) we get
\begin{align}\label{sumtot1}&\mathbf{E}\mathbf{Var}\bigg(\sum_{i=1}^{b^L}
 \frac {\Upsilon{(T_i)}}{\mu^{-2}{\ln^2{n_i}}}\Big|\Omega_L\bigg)=o\Big(\frac {{n^2}}{{\ln^4{n}}}\Big).\end{align}
Using (\ref{exptotal1,1}) and (\ref{sumtot}) and applying  (\ref{sumtot1}), the Chebyshev inequality results in that conditioning on $\Omega_L$,
\begin{align}\label{TPL1}&\sum_{i=1}^{b^L}\frac {\Upsilon{(T_i)}}{\mu^{-2}{\ln^2{n_i}}}=\sum_{i=1}^{b^L}
\frac{\alpha n_i}{\mu^{-1}\ln{n_i}}+\sum_{i=1}^{b^L}\frac{n_ir(n_i)}{\mu^{-2}{\ln^2{n_i}}} +o_p\Big(\frac {n}{{\ln^2{n}}}\Big).\end{align}
By applying Theorem \ref{christmas},  (\ref{hejsan}) and (\ref{hals}) we get 
\begin{align}\label{vit}\sum_{i=1}^{b^L}\frac{n_ir(n_i)}{\mu^{-2}{\ln^2{n_i}}}&=
\sum_{i=1}^{b^L}\frac{\zeta n_i}{\mu^{-2}{\ln^2{n_i}}}+\sum_{i=1}^{b^L}\frac {o(n_i)}{\ln^{2}{n_i}}
\nonumber\\&=\frac {\zeta n}{\ln^{2}{n}}+ o\left(\frac {n}{\ln^{2}{n}}\right).\end{align}

 Thus,  (\ref{TPL}) follows from (\ref{TPL1}) and (\ref{vit}).

\end{proof}



\subsubsection* {Acknowledgement:}
Professor Svante Janson is gratefully acknowledged for invaluable support and advice. I  also thank Dr Nicolas Broutin for helpful discussions.

\end{document}